\documentclass[10pt]{article}
\catcode`\@=11
\catcode`\@=12

\usepackage{graphicx}
\usepackage[centertags]{amsmath}
\usepackage{amsfonts,amssymb,amsthm}
\usepackage{mathrsfs}

\usepackage{color}

\newtheorem{theorem}{Theorem}[section]

\newtheorem{prop}[theorem]{Proposition}

\newtheorem{lemma}[theorem]{Lemma}

\newtheorem*{caveat*}{Caveat}

\newtheorem*{definitions*}{Definitions}

\theoremstyle{definition}

\newcommand{\dsr}{\mathscr D}
\newcommand{\kO}{\mathfrak O}
\newcommand{\Om}{\Omega}
\newcommand{\om}{\omega}
\newcommand{\ol}{\overline}

\renewcommand{\t}{\tau}
\newcommand{\tm}{\overline t}
\def\NN{\mathbb N}

\newcommand{\wk}{~\rightharpoonup~}

\renewcommand{\leq}{\leqslant}
\renewcommand{\geq}{\geqslant}

\def\RR{\mathbb R}

\def\div{\mathop{\mathrm{div}}}

\newcommand{\vt}{\vartheta}

\newcommand{\sE}{\mathcal{E}}
\newcommand{\sU}{\mathcal U}

\newcommand{\half}{{\frac{1}{2}}}

\newcommand{\sR}{\mathcal{R}}

\newcommand{\x}{x_1}
\newcommand{\y}{x_2}

\newcommand{\sS}{\mathcal{S}}

\newcommand{\CC}{\mathbb{C}}

\newcommand{\sP}{\mathcal{P}}

\newcommand{\ul}{\underline}

\newcommand{\nc}{\newcommand}

\nc{\comment}[1]{} \nc{\myref}[1]{{\rm(\ref{#1})\comment{#1}}}
\nc{\RLp}{\mbox{$L_{p}(\Gamma)\;$}}
\nc{\RLh}{\mbox{$L_{2}(\Gamma)\;$}} \nc{\sL}{\mathcal L}
\nc{\RLpr}{\mbox{$L_{p}(\Gamma,\rho)\;$}}
\nc{\RLhr}{\mbox{$L_{2}(\Gamma,\rho)\;$}} \nc{\ds}{\displaystyle}

\def\sin{\mathop{\rm sin}\nolimits}
\def\cos{\mathop{\rm cos}\nolimits}

\def\inf{\mathop{\rm inf}\nolimits}

\def\sup{\mathop{\rm sup}\nolimits}

\nc{\bsl}{\backslash} \nc{\T}{\Theta} \nc{\W}{|W^*|}

\nc{\la}{\label} \nc{\sB}{\mathcal B}
\nc{\mS}{\mathscr S}

\newcommand{\beq}{\begin{equation}}
\newcommand{\eeq}{\end{equation}}

\numberwithin{equation}{section}

\begin{document}

\parskip=6pt 
\parindent=0pt 
\title{On the stability of travelling waves with vorticity obtained  
by minimization} 
\author{}
\author{B. Buffoni
\footnote{Supported by a grant of the Swiss National Science Foundation.}\\ 
{\small Section de math\'ematiques}\\  
{\small \'Ecole Polytechnique F\'ed\'erale}\\  
\vspace{4mm}
{\small Lausanne, CH-1015}\\
G. R. Burton\\ 
{\small Department of Mathematical Sciences}\\  
{\small University of Bath}\\  
{\small Bath BA27 AY, UK}
} 
\date{
Revised January 2013 (earlier versions: September 2011 and July 2012)\\
\hspace{2cm}\\
{\em NoDEA Nonlinear Differential Equations and Applications}, to appear\\
The final publication will be available at http://link.springer.com\\
doi 10.1007/s00030-013-0223-4 }

\maketitle 

\begin{abstract}
We modify the approach of Burton and Toland \cite{BuTo}
to show the existence of periodic surface water waves with vorticity
in order that it becomes
suited to a stability analysis. This is achieved by enlarging
the function space to a class of stream functions that
do not correspond necessarily to travelling profiles. In particular,
for smooth profiles
and smooth 
stream functions, the normal component of the velocity field
at the free boundary is not required a priori to vanish in some
Galilean coordinate system. Travelling periodic waves are obtained by
a direct minimization of
a functional that corresponds to the total energy and that is therefore
preserved by the time-dependent evolutionary problem (this 
minimization appears in \cite{BuTo} after a first maximization).
In addition, we not only
use the circulation along the upper boundary as a constraint, but also the
total horizontal impulse (the velocity becoming a Lagrange multiplier).
This allows us to preclude parallel flows 
by choosing appropriately the values of these two constraints and the
sign of the vorticity.
By stability, we mean conditional energetic stability of the set
of minimizers as a whole,
the perturbations being spatially periodic of given period.
Our proofs depend on the assumption that the surface offers some resistance
to stretching and bending.
\end{abstract}

\section{Introduction} \label{the problem}

For a fixed H\"{o}lder exponent $\gamma\in(0,1)$, 
period $P>0$ and average height $Q>0$,
we shall consider domains $\mathit \Om\subset \RR^2$ and curves
$\mathscr S$ such that there
exists a $C^{1,\gamma}$-map $F:\RR^2\rightarrow \RR^2$ 
satisfying the following
properties:
\begin{itemize}
\item $F$ restricted to $\RR\times [0,Q]$ is a diffeomorphism
from $\RR\times [0,Q]$ onto $\overline{\mathit \Om}$,
\item meas$(\mathit \Om\cap((0,P)\times \RR))=PQ$,
\item $F(x_1,0)=(x_1,0)$ for all $x_1\in\RR$,
\item $\mathscr S\subset \RR\times(0,\infty)$ and $F$ 
 restricted to $\RR\times\{Q\}$ is a homeomorphism from
$\RR\times \{Q\}$ onto $\mathscr S$,
\item $F(x_1+P,x_2)=(F_1(x_1+P,x_2),F_2(x_1+P,x_2))
=(F_1(x_1,x_2)+P,F_2(x_1,x_2))$ for all $x=(x_1,x_2)\in \RR\times[0,Q]$.
\end{itemize}
As a consequence the
curve $ \mathscr S$ is of class $C^{1,\gamma}$ in the  open upper half plane, 
$P$-periodic and 
is a connected component of the boundary of
the region  $\mathit \Om$ . 
Let $\sS$ and $\Om$  denote one period of $\mathscr S$ and $\mathit \Om$.
We denote by $\kO$ the set of all domains $\Om$ defined in this way, and we
write $\Om\in \kO$ or $\mathit\Om\in \kO$.
Thus $\mathscr S$ must be a simple curve (without self-intersection or self-touching)
but it need not be the graph of a function.
While this is fairly general, it excludes some cases of physical interest.
For example, 
a row of rolling beads of mercury constitutes a travelling wave with
a disconnected free surface whose components are not graphs of functions,
and beads of mercury can touch without coalescing.

If $\RR^2$ is identified with the complex plane $\CC$, the point
$(x_1,x_2)$ corresponding to the complex number $x_1+ix_2$ ,
it can be shown (see e.g. the appendix A of the paper by Constantin and 
Varvaruca \cite{CoVa}) that there exists a holomorphic  map 
\begin{equation}
\label{eq: widetilde psi}
\widetilde \phi+i\widetilde \psi:\Om\rightarrow \RR\times (0,1)
\end{equation}
such that
\begin{itemize}
\item $\widetilde \phi+i\widetilde \psi$ can be extended into a diffeomorphism
from $\overline{\mathit \Om}$ onto $\RR\times [0,1]$,
\item $\widetilde \psi,\widetilde \phi$ are real-valued functions 
of class $C^{1,\gamma}$
on $\overline{\mathit \Om}$ and their gradients never vanish on 
$\overline \Om$,  
\item   
$\widetilde \psi|_{\{x_2=0\}}=0$ and $\widetilde \psi|_{\mathscr S}=1$,
\item $\widetilde \phi(x+P)+i\widetilde \psi (x+P)
=\widetilde \phi(x)+i \widetilde \psi(x)+\widetilde P$ 
for all $x=x_1+i x_2\in \RR\times[0,1]$, where

\begin{equation}
\label{eq: widetilde P}
\widetilde P=\int_0^P\partial_1\widetilde \phi(x_1,0)dx_1
=\int_0^P\partial_2\widetilde \psi(x_1,0)dx_1
=\int_{\sS}\nabla\widetilde \psi\cdot n\, dS
\end{equation}
and $n$ is the outward normal to $\Om$ at a point of $\sS$.
\end{itemize}
We shall write $\xi\in H^{1/2}_{per}(\sS)$ or
$\xi\in H^{1/2}_{per}(\mathscr S)$ if 
$\xi$ is the trace 
on $\mathscr S$
of some $\psi_\xi\in H^1_{loc}(\mathit \Om)$ that is $P$-periodic in $x_1$.
Analogously, we shall write $\zeta\in L^2_{per}(\mathit \Om)$ if  
$\zeta\in L^2_{loc}(\mathit \Om)$   is $P$-periodic 
in $x_1$.

 Given $\mathit \Om$, $\mathscr S$, $\xi\in H^{1/2}_{per}(\mathscr S)$ 
and $\zeta\in L^2_{per}(\mathit \Om)$,
let $\psi\in H_{\rm loc}^{1}(\mathit \Om)$ be the weak solution
of the boundary value problem
\begin{subequations}\label{fbvp}\begin{align}\label{2.1a}
&-\Delta \psi = \zeta \text{ on } \mathit \Om,\\
&\psi(x_1,0) = 0,\label{bottom bc}\\
&\psi=\xi \text{ on }\mathscr S,\label{const.}\\
&\psi \text{ is ${P}$-periodic in $\x$, written 
$\psi\in H^1_{per}(\Om)$ or $\psi\in H^1_{per}(\mathit \Om)$.}
\label{eq: H1per}
\end{align}

On one period, the circulation $C$
and the total horizontal impulse $I$ 
are given by
\begin{equation*}
C=C(\Omega,\xi,\zeta):=\int_{\sS} \nabla \psi\cdot n \, dS,
\end{equation*}
\begin{equation*}
I=I(\Omega,\xi,\zeta):=\int_{\Om} \partial_2 \psi dx
=\int_{\Om}\nabla x_2 \cdot \nabla  \psi dx.
\end{equation*}
By $C(\Omega,\xi,\zeta)=\int_{\sS} \nabla \psi\cdot n \, dS$, we mean
$$C(\Omega,\xi,\zeta)=\int_{\Om} \nabla \psi\cdot \nabla \widehat \psi\,dx
-\int_{\Om} \zeta \widehat \psi\,dx,$$
where $\widehat\psi$ is any function in $H^1_{per}(\Om)$ such that
$\widehat \psi|_{\{x_2=0\}}=0$ and 
$\widehat \psi|_{\sS}=1$. For example we can choose 
$\widehat\psi=\widetilde\psi$.
When $\psi$ is regular enough, these
two ways of defining $C(\Om,\xi,\zeta)$ agree, but the latter one requires
less regularity. We can also write, if there is enough regularity available,
$$I(\Omega,\xi,\zeta)
=\int_{\sS} x_2\nabla \psi\cdot n  \,dS +\int_{\Om} x_2\zeta\, dx.$$

Let us fix $\mu$ and $\nu$ in $\RR$.
Then 
$(\Omega,\xi,\zeta)$
defines a travelling water wave with stream function $\psi$, 
circulation $\mu$, total horizontal impulse $\nu$
and vorticity $\zeta$,   if, in addition,
\begin{equation}
C(\Om,\xi,\zeta)=\mu,~~I(\Om,\xi,\zeta)=\nu,\label{circ}
\end{equation}
\begin{equation}
\label{special xi}
\xi=\lambda_1 x_2+\lambda_2 |_{\mathscr S} \text{ for some }\lambda_1,\lambda_2\in\RR,
\end{equation}
\begin{equation}
\zeta = \lambda \circ (\psi-\lambda_1 x_2) 
\text{ almost everywhere for some function $\lambda$}
\label{vorticity}\\
\end{equation}
and
\begin{equation}
\half | \nabla \psi -(0,\lambda_1) |^2  +g\,x_2
=
\text{  constant on $\mathscr S$,}
\label{bern}
\end{equation}
where $g$ is gravity. The travelling wave is moving with speed $\lambda_1$ to the right and
 equation \eqref{vorticity} reflects the fact that vorticity in steady flows is constant on streamlines.
The constants $\lambda_1,\lambda_2$ in \eqref{special xi} and the function 
$\lambda$ in \eqref{vorticity}  
 are not prescribed.

If the surface reacts to
stretching and bending, the Bernoulli condition
\eqref{bern} is replaced by
 \begin{multline}
 \label{bern''}\half | \nabla \psi -(0,\lambda_1) |^2
  +g\,x_2  - T \beta \big(\ell(\sS)-P\big)^{\beta -1}\sigma\\
 + E\Big(2\sigma '' +\sigma^3\Big) =
\text{  constant on $\mathscr S$,}\tag{\ref{bern}$'$}
 \end{multline}
 where $'$ denotes differentiation with respect to arc length along the surface,
$\sigma(x)$ is the curvature of the surface at $x\in \mathscr S$,
 $\ell(\sS)$ is the length of $\sS$,  $E\geq 0$ is a coefficient of bending resistance and $\beta \geq 1$.
See \cite{Tol-hydro}.
The case $E=0$ and $\beta=1$ corresponds to simple surface tension
with coefficient $T$.
\end{subequations}

The total energy $\sL(\Om,\xi,\zeta)$ of  a solution of 
\eqref{fbvp}(a--d) in one period is the
sum of the kinetic energy,
the gravitational potential energy
and the surface energy: 
\begin{equation}\label{eq: definition of sL}
\sL (\Om, \xi,\zeta):=\frac 1 2  \int_\Om
|\nabla \psi|^2\,  d x + g \int_\Om x_2 ~d x +\sE(\sS),
\end{equation}
where $\psi$ is the solution to the corresponding boundary
value problem \eqref{fbvp}(a--d),
\begin{equation}\label{eq: definition of sE}
\sE(\sS)=
T (\ell(\sS)-P)^\beta +E
 \int_{0}^{\ell(\sS)}|\sigma|^2 ds,
\end{equation}
and $s$ is the arc length.
\footnote{If $p$ is a parametrisation of $\mathscr S$
 such that $|\frac d {dx}p|$
is constant and $p(x+P)=p(x)+(P,0)$, then
 $$\int_{0}^{\ell(\sS)}|\sigma|^2 ds=
\left(\frac{P}{\ell(\sS)}\right)^3
 \int_{0}^{P}\left|\frac{d^2}{ dx^2}p(x)\right|^2 dx.$$
In \cite{BuTo}, the power $3$ is wrongly omitted in several places, without invalidating the main results.}
Hence we are lead to the minimization problem
$$
\min \{\sL (\Om,\xi,\zeta):
\Om\in \kO,\, \xi \in H^{1/2}_{per}(\sS),\zeta \in \sR(\Om),
 C=\mu, I=\nu\},
$$
where $\kO$ is the class of domains $\Om$ described above and
$\sR(\Om)\subset L^2(\Om)$  is the set of rearrangements supported in $\Om$ of a given function 
$\zeta_Q\in L^2(\Om_Q)$, where
$\Om_Q=(0,P)\times (0,Q)$. 
Note that $\Om\neq \Om_Q$ is allowed and $\zeta_Q$ 
does not depend on $\Om$.
However, in general,  $\sR(\Om)$ is not weakly closed in $L^2(\Om)$
and we shall work instead with its weak closure
$\ol{\sR(\Om)}^w$ in 
$L^2(\Om)$, which is a convex subset of $L^2(\Om)$; see the discussion in
\cite[p. 979, 3rd parag.]{BuTo}.
Hence, as in \cite{BuTo}, we shall rather consider
\begin{equation}\label{min}
\min \{\sL (\Om,\xi,\zeta):
\Om\in \kO,\, \xi \in H^{1/2}_{per}(\sS),\zeta \in \ol{\sR(\Om)}^w,
 C=\mu, I=\nu\}.
\end{equation}

Observe that $\mathit \Om_Q := \RR \times (0,Q)\in \kO$.
We  write $\Om \in \kO$ or $\mathit \Om \in \kO$, and we assume that
$\sL(\Om,\xi,\zeta)=+\infty$ is allowed, for example if 
the surface energy is infinite because the boundary is not regular enough. 
We assume $T>0$, $\beta \geq 1$ and $E>0$  in order to obtain compactness
for the above minimization problem.

In \eqref{min},
the boundary condition \eqref{special xi} is not prescribed, but we will
show that it holds for minimizers.
Hence, in \eqref{min}, any stream function $\psi$
that is compatible with the vorticity function $\zeta$ is allowed (by
choosing $\xi=\psi|_{\mathscr S}$).
This feature will be crucial in the stability analysis of section 
\ref{section: stability}.

{\bf A way of avoiding parallel flows.}
When $\Om=\Om_Q$,
by taking $\widehat\psi=x_2/Q$ we get
$$I(\Om,\xi,\zeta)=Q\int_{\Om}\nabla( x_2/Q)\cdot \nabla \psi dx=
QC(\Om,\xi,\zeta)+\int_\Om x_2\zeta \, dx.$$
Hence, if $\zeta_Q$ is essentially one-signed and not trivial,
then $I(\Om_Q,\xi,\zeta)-QC(\Om_Q,\xi,\zeta)$ $\neq 0$ has the same sign as $\zeta_Q$. 
Thus,
to avoid parallel flows, it seems natural to choose
$\mu,\nu$ so that  $(\nu-Q\mu)\zeta_Q\leq 0$ a.e.
(or $\nu-Q\mu\neq 0$ if $\zeta_Q$ vanishes a.e.).

In \cite{BuTo}, parallel flows were precluded by choosing $\mu$ large enough.
They were proved to be saddle points of the energy, and thus different
from any minimizer (there, the energy functional was obtained 
after a first maximization).
For  related works on global minimization in hydrodynamical problems
and stability,
see \cite{CoSaSt,CoSt,Bu,grbarma,ChSeZa}.  In particular, the paper
\cite{CoSaSt} by Constantin, Sattinger and Strauss contains two variational 
formulations for gravity water waves with vorticity. In their first
formulation, instead of considering the constraint $\zeta\in 
\ol{\sR(\Om)}^w$ for a given 
$\zeta_Q\in L^2(\Om_Q)$ (among other constraints), 
they subtract
from the energy functional 
a term of the form $\int_{\Om} F(\zeta)dx$, where
$F:\RR\rightarrow \RR$ is a given $C^2$-function such that $F''$
never vanishes. As a result, for any critical point, $(F')^{-1}$
turns out to be the so-called vorticity function.
They do not apply their approach to existence results, but it
leads to  an elegant linear stability analysis
in \cite{CoSt}.

{\bf Overview of the paper.}
 Section 2 discusses the solution by minimization of
the elliptic equation $-\Delta\psi=\zeta$ for fixed $\Omega$, $\zeta$, $\mu$
and $\nu$ and establishes the unknown boundary data $\xi$.
In Section 3 it is shown that that the Bernoulli boundary condition is 
satisfied by constrained 
minimizers when $\Omega$ is allowed to vary.
Section 4 proves compactness of minimizing sequences and 
establishes the existence of constrained minimizers.
The main stability result is Theorem 5.2 which is proved using compactness of 
minimizing sequences together with some theory of transport equations
summarised in the Appendix.

{\bf Some open questions.}
\newline
-- Is there a criterion that ensures uniqueness of the constrained minimizer
(up to translational invariance)?
In such a case, the present notion of stability would be related
to ``orbital'' stability.
\newline
-- If $\zeta_Q$ is smooth, what can be said about
the regularity of the minimizers?
\newline
-- Is there an explicit $\zeta_Q$ for which the free boundaries
of the minimizers are not graphs?
\newline
-- For an initial profile near the one of a minimizer, is 
the solution to the evolutionary problem defined for  small enough 
positive times?
A stability result like Theorem 5.2 stated under this assumption 
is qualified as 
``conditional'' (see \cite{Mielke} and, for  well-posedness
issues for related settings, see e.g. \cite{CoSc}).
We therefore raise the question
whether such a solution to the evolutionary problem is defined for all 
positive times.

\section {Minimization on fixed domain}
We begin with a useful lemma.
\begin{lemma}\label{lemma: on lambda_i}
Suppose that $\Om\in \kO\backslash \{\Om_Q\}$ and $\zeta\in L^2(\Om)$. 
Then 
$$C(\Om,1,0)=\int_{\Om}|\nabla \widetilde \psi|^2dx>P/Q$$
(see \eqref{eq: widetilde psi} for the definition of 
$\widetilde \psi$) and,
for all $\mu,\nu\in\RR$, there exist $\lambda_1=\lambda_{1,\Om,\zeta}$ and 
$\lambda_2=\lambda_{2,\Om,\zeta}$ such that
$$C(\Om, \lambda_1x_2+\lambda_2,\zeta)=\mu,~~
I(\Om, \lambda_1x_2+\lambda_2,\zeta)=\nu.~~$$
Moreover $\lambda_1,\lambda_2\in\RR$ are unique.
\end{lemma}

\begin{proof}
We require
\begin{eqnarray*}
\mu=C(\Om,\lambda_1 x_2+\lambda_2,\zeta)=\lambda_2 C(\Om,1,0)
+\lambda_1 C(\Om,x_2,0)+C(\Om,0,\zeta)
\\=\lambda_2 C(\Om,1,0)+\lambda_1 P+C(\Om,0,\zeta)
,\\
\nu=I(\Om,\lambda_1x_2+\lambda_2,\zeta)=\lambda_2 I(\Om,1,0)
+\lambda_1 I(\Om,x_2,0)+I(\Om,0,\zeta)
\\=\lambda_2 P+\lambda_1 PQ,\\
\end{eqnarray*}
because
$$C(\Om,x_2,0)
=\int_{\sS}\nabla x_2 \cdot n \, dS
=\int_{\Om}\text{div}(\nabla x_2) dx
+\int_{0}^P\partial_2 x_2\, dx_1
=\int_{0}^P\partial_2 x_2\, dx_1=P,$$
$$I(\Om,x_2,0)=\int_{\Om}\partial_2 x_2 dx=PQ,$$
\begin{eqnarray*}
&&I(\Om,1,0)=\int_\Om \nabla x_2\cdot\nabla \widetilde \psi\, dx
=\int_\Om \div(\widetilde \psi\nabla x_2) dx
=\int_\Om \div\Big((\widetilde \psi-1)\nabla x_2\Big) dx
\\&&=\int_{\partial \Om} (\widetilde \psi-1)\nabla x_2\cdot n\,  dS
=\int_{0}^P\partial_2 x_2\, dx_1=P
\end{eqnarray*}
and
$$I(\Om,0,\zeta)=\int_\Om \nabla x_2\cdot\nabla \psi\, dx
=\int_\Om \div(\psi\nabla x_2) dx
=\int_{\partial \Om} \psi\nabla x_2\cdot n\,  dS=0,$$
where $\psi$ is the solution to the system 
\eqref{2.1a} to \eqref{eq: H1per}
with $\xi=0$.

Let $\widetilde \psi$ be, as in
\eqref{eq: widetilde psi}, the harmonic function on $\mathit \Om$ that
vanishes on $\{x_2=0\}$, is $1$ on $\mathscr S$ and is $P$-periodic
in $x_1$. Then, by \eqref{eq: widetilde P}, $\widetilde P
=C(\Om,1,0)=\int_\Om |\nabla \widetilde\psi|^2dx$. Let us check that
\begin{equation}\label{eq: non-degeneracy}
C(\Om,1,0)\geq P/Q
\text{ with equality exactly when } \Om=\Om_Q.
\end{equation}
In order to do this, consider as in
\eqref{eq: widetilde psi}
 the harmonic conjugate
$\widetilde \phi$ of $\widetilde \psi$, that is,
 $\nabla \widetilde \phi$
is obtained from $\nabla \widetilde \psi$ by a clockwise rotation
through $\pi/2$. 
Then $\widetilde \phi(x+P)-\widetilde \phi(x)$ is a constant
equal to $\widetilde P=C(\Om,1,0)$ (see above) and
the map $(\widetilde \phi,\widetilde \psi)$ is a diffeomorphism
from $\mathit \Om$ to $\RR\times (0,1)$.

We denote by $(u,v)$ the 
Euclidean coordinates in $\RR\times (0,1)$ and by $(u,v)\rightarrow x_2(u,v)$
the map that associates
with $(u,v)$ the $x_2$ coordinate of
the corresponding point in $\mathit\Omega$.
Observe that
\begin{equation}
\partial_{u,v}x_2
=(\partial_u x_2,\partial_v x_2)
= \partial_{x_1,x_2}x_2\,   (\partial(x_1,x_2)/\partial(u,v)) $$
(Jacobian matrix),
$$\partial_{u,v}v=  \partial_{x_1,x_2}\tilde \psi\,   (\partial(x_1,x_2)/\partial(u,v)) ~, $$
 $$(\partial(x_1,x_2)/\partial(u,v))  (\partial(x_1,x_2)/\partial(u,v))^T
= \{ det  (\partial(x_1,x_2)/\partial(u,v)) \} I
\label{chvar}
\end{equation}
(multiple of the identity matrix; this is a consequence of the Cauchy-Riemann equations)
and thus
$$   \partial_{u,v}x_2  \cdot  \partial_{u,v}x_2 
=   \partial_{x_1,x_2}x_2  \cdot   \partial_{x_1,x_2}x_2\,     det  (\partial(x_1,x_2)/\partial(u,v)) 
=    det  (\partial(x_1,x_2)/\partial(u,v)) $$
and
$$   \partial_{u,v}x_2  \cdot  \partial_{u,v}v 
=   \partial_{x_1,x_2}x_2  \cdot   \partial_{x_1,x_2}\tilde \psi \,    det  (\partial(x_1,x_2)/\partial(u,v)) $$

As a consequence, we get that 
$$\int_{u=0}^{\widetilde P}\int_{v=0}^1|\nabla x_2(u,v)|^2dudv=\int_{\Om}dx=PQ$$
and
\begin{multline*}
\int_0^{\widetilde P} \partial_2 x_2 (u,1)du\stackrel{\text{Gauss}}{=}
\int_0^{\widetilde P}\int_0^1\div(v\nabla x_2(u,v)) dudv
\\=\int_0^{\widetilde P}\int_0^1\nabla x_2(u,v)\cdot \nabla v \, dudv
=\int_\Om \partial_{x_1,x_2} x_2\cdot 
\partial_{x_1,x_2} \widetilde \psi \, dx=I(\Om,1,0)=P.
\end{multline*}
Hence
\begin{multline*}
PQ\geq \min\Big\{
\int_0^{\widetilde P}\int_0^1|\nabla y(u,v)|^2dudv:
y\in H^1_{per}((0,\widetilde P)\times(0,1)), y(\cdot,0)=0,
\\\int_0^{\widetilde P}\partial_2y(u,1)du=P\Big\}.
\end{multline*}
The minimum depends on $\tilde P$ and therefore
it depends on the shape of the domain $\Omega$, because 
$\tilde P=C(\Omega,1,0)$.
The minimum is reached exactly at the function $y(u,v)=(P/ \tilde P) v$, which
shows that the value of the minimum 
is $(P/\tilde P)^2 \tilde P=P^2/C(\Omega,1,0)$. Hence
$PQ\geq P^2/C(\Omega,1,0)$ and $C(\Omega,1,0)\geq P/Q$ with equality exactly 
when $\Omega=\Omega_Q$.
Since $\Om\neq \Om_Q$ we now have $QC(\Om,1,0)-P>0$, so the equations
for $\lambda_1$ and $\lambda_2$ can be solved uniquely.
\end{proof}

\begin{prop}
\label{prop: fixed omega and zeta}
Given $\Om\in\kO\backslash\{\Om_Q\}$,
$\zeta\in L^2 (\Om)$ and $\mu,\nu\in\RR$,
the minimizer $\xi_{\Om,\zeta}$ for
the kinetic energy over
$\{\xi \in H^{1/2}_{per}(\sS):C(\Om,\xi,\zeta)=\mu,I(\Om,\xi,\zeta)=\nu\}$
exists and is unique, and 
there exist $\lambda_1$ and $\lambda_2$ in $\RR$ such that
\begin{equation}
\label{eq: def best xi}
\xi=\xi_{\Om,\zeta}=(\lambda_{1}x_2+\lambda_{2})|_{\sS}~.
\end{equation}
\end{prop}
\begin{proof}

We consider the minimum of the functional
$\psi\rightarrow \half \int_\Om |\nabla \psi|^2 dx$
over $\psi\in H^{1}_{per}(\Om)$ such that
$$-\Delta \psi=\zeta~\text{ on }~\Omega,~\psi(\cdot,0)=0,$$
$$\int_{\Om}\nabla \psi\cdot \nabla \widetilde \psi\, dx 
-\int_{\Om}\zeta \widetilde \psi\, dx =\mu
~\text{ and }~
\int_{\Om}\nabla \psi\cdot\nabla x_2 \, dx =\nu,
$$
where $\widetilde \psi$ is defined in \eqref{eq: widetilde psi}
(such a  $\psi$ exists, by Lemma  \ref{lemma: on lambda_i}).
A standard convexity argument gives a
minimizer $\psi$ and it suffices to set $\xi=\psi|_{\sS}$.

Consider any $h\in H^1_{per}(\Om)$
such that
$$\Delta h=0,~h|_{\{x_2=0\}}=0,~\int_{\Om}\nabla h\cdot\nabla x_2 \, dx =0\text{ and }
\int_{\Om}\nabla h\cdot\nabla \widetilde \psi \, dx =0.$$
For all $t\neq 0$, we get
$$\frac 1 2 \int_\Om|\nabla \psi|^2dx 
\leq \frac 1 2 \int_\Om|\nabla (\psi+th)|^2dx 
= \frac 1 2 \int_\Om|\nabla \psi|^2dx 
+  t \int_\Om\nabla \psi\cdot\nabla h\, dx 
+ \frac 1 2 t^2  \int_\Om|\nabla h|^2dx $$
and thus
$\int_\Om\nabla \psi\cdot\nabla h\, dx=0$. 
More generally, if $h\in H^1_{per}(\Om)$ only satisfies $\Delta h=0$ and
$h|_{\{x_2=0\}}=0$, we consider instead of $h$ the function
\begin{multline*}
h-\frac{P\int_\Om \nabla h\cdot\nabla \widetilde \psi\, dx-
\int_\Om|\nabla \widetilde \psi|^2\, dx
\int_\Om \nabla h\cdot \nabla x_2 \, dx}{P^2-PQ
\int_\Om|\nabla \widetilde \psi|^2\, dx}x_2 \\
-\frac{P\int_\Om \nabla h\cdot\nabla x_2\, dx-
PQ\int_\Om \nabla h\cdot \nabla \widetilde \psi\, dx}{P^2-PQ
\int_\Om|\nabla \widetilde \psi|^2\, dx}\widetilde\psi,
\end{multline*}
which satisfies the two additional constraints,
in view of the relations
$$
\int_\Omega \nabla x_2 \cdot \nabla \widetilde\psi\,  dx= P; \quad\quad\quad
\int_\Omega |\nabla x_2|^2 dx = PQ.
$$
Instead of $0=\int_\Om\nabla \psi\cdot\nabla h\, dx$, we get 
\begin{multline*}
0=\int_\Om \nabla \psi\cdot \nabla h\, dx-\frac{P\int_\Om \nabla h\cdot\nabla \widetilde \psi\, dx-
\int_\Om|\nabla \widetilde \psi|^2\, dx
\int_\Om \nabla h\cdot \nabla x_2 \, dx}{P^2-PQ
\int_\Om|\nabla \widetilde \psi|^2\, dx
}\int_\Om
\nabla x_2\cdot \nabla \psi\, dx\\
-\frac{P\int_\Om \nabla h\cdot\nabla x_2\, dx-
PQ\int_\Om \nabla h\cdot \nabla \widetilde \psi\, dx}{P^2-PQ
\int_\Om|\nabla \widetilde \psi|^2\, dx
}
\int_{\Om}\nabla \widetilde\psi\cdot\nabla \psi\, dx\\
=\int_\Om \nabla \psi\cdot \nabla h\, dx
+\frac{
\int_\Om|\nabla \widetilde \psi|^2\, dx
\int_\Om \nabla x_2\cdot\nabla \psi\, dx-
P\int_\Om \nabla \widetilde \psi\cdot \nabla \psi \, dx}
{P^2-PQ
\int_\Om|\nabla \widetilde \psi|^2\, dx
}\int_\Om
\nabla x_2\cdot \nabla h\, dx\\
+\frac{PQ\int_\Om \nabla \widetilde \psi\cdot\nabla \psi\, dx-
P\int_\Om \nabla x_2\cdot \nabla \psi\, dx}{P^2-PQ
\int_\Om|\nabla \widetilde \psi|^2\, dx
}
\int_{\Om}\nabla \widetilde\psi\cdot\nabla h\, dx
\\=\int_\Om \nabla\Big\{ \psi
+\frac{
\int_\Om|\nabla \widetilde \psi|^2\, dx
\int_\Om \nabla x_2\cdot\nabla \psi\, dx-
P\int_\Om \nabla \widetilde \psi\cdot \nabla \psi \, dx}
{P^2-PQ
\int_\Om|\nabla \widetilde \psi|^2\, dx
}x_2\\
+\frac{PQ\int_\Om \nabla \widetilde \psi\cdot\nabla \psi\, dx-
P\int_\Om \nabla x_2\cdot \nabla \psi\, dx}{P^2-PQ
\int_\Om|\nabla \widetilde \psi|^2\, dx
}
\widetilde\psi\Big\}
\cdot \nabla h\, dx
\end{multline*}
for all  $h\in H^1_{per}(\Om)$ such that $\Delta h=0$ and $h|_{\{x_2=0\}}=0$.
Hence, as we explain below,
there exist   $\lambda_1$  and $\lambda_2$ in $\RR$ satisfying
\eqref{eq: def best xi},
namely
$$\lambda_1=\lambda_{1,\Om,\zeta}
=-
\frac{
\int_\Om|\nabla \widetilde \psi|^2\, dx
\int_\Om \nabla x_2\cdot\nabla \psi\, dx-
P\int_\Om \nabla \widetilde \psi\cdot \nabla \psi \, dx}
{P^2-PQ
\int_\Om|\nabla \widetilde \psi|^2\, dx
}$$
and
$$\lambda_2=\lambda_{2,\Om,\zeta}
=-\frac{PQ\int_\Om \nabla \widetilde \psi\cdot\nabla \psi\, dx-
P\int_\Om \nabla x_2\cdot \nabla \psi\, dx}{P^2-PQ
\int_\Om|\nabla \widetilde \psi|^2\, dx
}
$$
Observe that these values must be equal to those obtained in Lemma
\ref{lemma: on lambda_i}, but here they are expressed with the help of the
minimal stream function $\psi$.
Hence the uniqueness statement in
Lemma \ref{lemma: on lambda_i} gives the desired uniqueness of the minimizer
$\xi$.

Let us briefly explain why 
$\psi-\lambda_1 x_2-\lambda_2\widetilde\psi=0$ on $\mathscr S$ if
$$
\int_\Om\nabla(\psi-\lambda_1 x_2-\lambda_2\widetilde\psi)\cdot\nabla h\, dx=0
$$
for all  $h\in H^1_{per}(\Om)$ such that $\Delta h=0$ and $h|_{\{x_2=0\}}=0$.
Consider the holomorphic map $\widetilde \phi+i\widetilde \psi$ 
in \eqref{eq: widetilde psi}
and write $\psi=\psi_0\circ (\widetilde\phi+i\widetilde\psi)$ and
$h=h_0\circ (\widetilde\phi+i\widetilde\psi)$. We also use the notation
$(u,v)$ for the coordinates in $(0,\widetilde P)\times (0,1)$
and $(u,v)\rightarrow x_2(u,v)$ for the map that
associates with $(u,v)$
the $x_2$ coordinate of the corresponding point in $\Omega$. We get
$$\int_{0}^{\widetilde P}
\int_0^1\nabla(\psi_0(u,v)-\lambda_1 x_2(u,v)-\lambda_2v)
\cdot\nabla h_0(u,v)\, du dv=0$$
for all  $h_0\in H^1_{per}((0,\widetilde P)\times (0,1))$ such that 
$\Delta h_0=0$ and $h_0|_{\{v=0\}}=0$,
changing variables with the aid of \eqref{chvar}.
The upper boundary $\{v_2=1\}$ being regular, we can deduce that
$\psi_0(u,1)-\lambda_1 x_2(u,1)-\lambda_2=0$ for almost all $u$.
\end{proof}

equation \eqref{eq: def best xi} is a weak formulation of the condition that
the modified velocity field $(\partial_2\psi-\lambda_1 ,-\partial_1\psi)$ be 
tangent to  the upper boundary and correspond to a stationary wave 
that travels with speed $\lambda_1$ to the right.
This tangency condition would hold classically if the free surface were
of class $C^2$.
However our existence theorem in Section 4
below does
not yield enough regularity for this to be asserted at present.

\begin{prop}
\label{prop: fixed omega}
Let  $\Om\in \kO\backslash\{\Om_Q\}$ be given and let
$(\Om,\xi,\zeta)$  be a minimizer of $\sL$ over all
$(\Om,\widetilde \xi,\widetilde \zeta)$ such that
$\sL(\Om,\xi,\zeta)<\infty$,
 $\widetilde \xi\in  H^{1/2}_{per}(\mathscr S)$, 
$\widetilde  \zeta \in\ol{\sR(\Om)}^w$, 
$C(\Om,\widetilde \xi,\widetilde \zeta)=\mu$ and
$I(\Om,\widetilde \xi,\widetilde \zeta)=\nu$.

Then
there exist   $\lambda_1$  and $\lambda_2$ in $\RR$ such that
$\xi=(\lambda_{1}x_2+\lambda_{2})|_{\sS}$
and a decreasing function $\lambda$ such that 
$$
\zeta=\lambda\circ  (\psi-\lambda_{1} x_2) 
\text{ a.e. on } \Omega,
$$
where $\psi$ is the 
stream function related to
$(\Om,\xi,\zeta)$.

If $\zeta_Q$ is essentially one-signed then $\zeta \in \mathcal{R}(\Omega)$.
\end{prop}
{\bf Remark.}
Proposition \ref{prop: fixed omega} contains no assertion concerning
existence of minimizers.
Sufficient conditions for their existence will be given later.

\begin{proof}
Only the last statement need be proved.
For $h\in L^2(\Om)$ define $\psi_h\in H^{1}_{per}(\Omega)$  by 
$$-\Delta \psi_h =h,$$ 
$$ \psi_h=0 \text{ on }\{x_2=0\},$$
$$ \psi_h|_{\sS} \hbox{ is a linear combination of $1$ and }  x_2,$$ 
$$\mu=  \int_{\sS}  \nabla \psi_h \cdot n
\, dS,~~
\nu=\int_{\Om}\partial_2  \psi_h   \, dx.$$ 
Because $\Om\neq \Om_Q$ it follows that
$\psi_h$ is well defined and
$\psi_h|_{\sS}= \lambda_{1,\Omega,h} x_2+\lambda_{2,\Omega,h}$
in terms of the unique constants given by Lemma \ref{lemma: on lambda_i}.
In particular we take
$\lambda_1=\lambda_{1,\Om,\zeta}$,
$\lambda_2=\lambda_{2,\Om,\zeta}$ and observe that
$\psi_{\zeta}|_{\sS}$ is equal to the optimal 
$\xi_{\Om,\zeta}$ of Proposition \ref{prop: fixed omega and zeta}.
Then $\xi=\xi_{\Om,\zeta}$
and, for fixed $\Omega$,
 $\zeta$ minimizes the function 
$$h\rightarrow \frac 1 2 \int_{\Omega}|\nabla  \psi_{h}|^2dx$$ 
over all $h\in L^2(\Om)$ such that $h$  is
in $\ol{\sR(\Om)}^w$. 
As in \cite{BuTo}, for such a $h$ and all $t\in[0,1]$, we set 
$h_t=(1-t)\zeta+th \in \ol{\sR(\Om)}^w$ and get that 
$\psi_{h_t}=(1-t)\psi_{\zeta}+t\psi_h$ and that 
\begin{multline*} 
0\leq\frac 1 2 \int_{\Omega}|\nabla \psi_{h_t}|^2dx 
- \frac 1 2 \int_{\Omega}|\nabla \psi_{\zeta}|^2dx 
=t\int_{\Omega}\nabla(\psi_h-\psi_{\zeta}) 
\cdot \nabla \psi_{\zeta}\, dx+o(t) 
\\=t\int_{\Omega}\nabla(\psi_h-\psi_{\zeta}) 
\cdot \nabla (\psi_{\zeta}-\lambda_{1}x_2
-\lambda_{2}\widetilde\psi)\, dx
\\+t\lambda_{2}\int_{\Omega}\nabla(\psi_h-\psi_{\zeta}) 
\cdot \nabla \widetilde\psi\, dx
+t\lambda_{1}\int_{\Omega}\nabla(\psi_h-\psi_{\zeta}) 
\cdot \nabla x_2 \, dx+o(t) 
\\=t\int_{\Omega}(h-\zeta) 
(\psi_{\zeta}-\lambda_{1}x_2
-\lambda_{2}\widetilde\psi)\, dx
+t\lambda_{2}\int_{\Omega}(h-\zeta)\widetilde\psi \, dx
+o(t) 
\\=t\int_{\Omega}(h-\zeta)(\psi_{\zeta}-\lambda_{1} x_2)dx
+o(t) 
\end{multline*} 
 because 
$$(\psi_{\zeta}-\lambda_{1}x_2-\lambda_{2}\widetilde\psi)
|_{\partial\mathit\Om}=0,$$
$$
\int_{\Omega}\nabla(\psi_h-\psi_{\zeta}) \cdot\nabla \widetilde\psi \, dx
-\int_{\Omega}(h-\zeta)\widetilde \psi\, dx
=C(\Om,\psi_h|_{\sS},h)-C(\Om,\psi_\zeta|_{\sS},\zeta)=0,$$
$$ \int_{\Omega}\nabla(\psi_h-\psi_{\zeta}) \cdot \nabla x_2 \, dx
=I(\Om,\psi_h|_{\sS},h)-I(\Om,\psi_\zeta|_{\sS},\zeta)=0.$$

Hence 
$\int_{\Omega}(h-\zeta)(\psi_{\zeta}
-\lambda_{1} x_2)dx\geq 0$
and the map  
$$h\rightarrow 
\int_{\Omega}h(\psi_{\zeta}-\lambda_{1} x_2)dx$$ 
reaches its minimum at $\zeta$, where   $h\in\ol{\sR(\Om)}^w$. 
 As moreover $-\Delta (\psi_{\zeta}-\lambda_{1} x_2)
=\zeta$, 
the same argument as in \cite[Lemma 2.3]{BuTo} ensures that  
there exists a decreasing function $\lambda$ such that 
$$
\zeta=\lambda\circ  (\psi_{\zeta}-\lambda_{1,\Om,\zeta} x_2) 
\text{ a.e. on } \Omega.
$$
If $\zeta_Q$ is one-signed except on a set of zero measure then it follows
as in \cite[Lemma 2.3]{BuTo} that $\zeta \in \sR(\Om)$.
\end{proof}

\section{The Bernoulli Boundary Condition}
In what follows, we consider some fixed minimizer $(\ul\Om,\ul\xi,\ul \zeta)$
and outline how to adapt the method in \cite{BuTo}
 to show that the Bernoulli condition
 \eqref{bern} or \eqref{bern''} holds in some weak sense.
Let $\lambda_{1,\underline\Omega,\underline\zeta}$, $\lambda_{2,
\underline\Omega,\underline\zeta}$ and $\lambda$
be the constants and decreasing function given by Proposition
\ref{prop: fixed omega}.

\begin{theorem}
Suppose that the upper boundary $ \ul{\mathscr S}$ of $\ul{\mathit{ \Om}}$ 
is given by 
an $H^{2}$ regular curve and
$$\ul \Om\neq \Om_Q.$$
We set $\psi_0=\underline \psi-\lambda_{1,\ul\Om,\ul \zeta}  x_2$
and we  let
$p:\RR\rightarrow \RR^2$ such that $|p'(s)|=1$ on $\RR$ be an
$H^2$-parametrisation of $\ul{\mathscr S}$.
Then, for all solenoidal smooth vector fields $\omega$ 
defined in a neighbourhood of  $\overline{\ul{\mathit \Om}}$,
 vanishing on $\{x_2=0\}$ 
and $P$-periodic in $x_1$, 
any minimizer $(\ul\Om,\ul\xi,\ul\zeta)$ satisfies
\begin{multline*}
0=\int_{\ul{ \Omega}} \nabla \psi_0\cdot D\om 
\nabla \psi_0 dx
+g\int_{\ul{ \Omega}}\nabla\cdot(x_2\om)dx
\\+\beta T(\ell(\sS)-P)^{\beta-1}\int_{0}^{\ell(\ul\sS)}(\om\circ p)'(s)\cdot  p'(s)ds
\\+E\int_{0}^{\ell(\ul\sS)}\left(2(w\circ p)''\cdot p''
-3|p''|^2(\om\circ p)'\cdot p'\right)ds.
\end{multline*}
If $p$ and $\psi_0$ are regular enough, this can be written
\begin{multline*}
0=\int_{\ul \sS}\left(\frac 1 2 |\nabla \psi_0|^2+\Lambda
(\psi_0)\right)(\om\cdot n)dS
+g\int_{\ul \sS}x_2(\om\cdot n) dS
\\-\beta T(\ell(\sS)-P)^{\beta-1}\int_{\ul\sS}
\sigma(\om\cdot n)dS
+E\int_{\ul\sS}(\sigma^3+2\sigma'')(\om\cdot n)dS
\end{multline*}
where $\Lambda$ is a primitive of $\lambda$ and $\sigma$ is
the curvature, and thus
$$
\frac 1 2 |\nabla \psi_0|^2+gx_2
-\beta T(\ell(\sS)-P)^{\beta-1}
\sigma+E(\sigma^3+2\sigma'')$$
is constant on $\ul{\mathscr S}$.
\end{theorem}

\begin{proof}
We only explain how to get the term 
$$\int_{\ul{\mathit \Omega}} \nabla \psi_0\cdot D\om 
\nabla \psi_0 dx$$
by following the method of
\cite[Subsection 2.3]{BuTo}, since the other terms do not
involve $\psi_0$ so the calculations are the same as in \cite{BuTo}.
For small $t \geq 0$ let the diffeomorphims $\tau$ be defined on $\Omega$
by $\tau(t)(x)=X(t)$, where
$$\dot X(t)=\om(X(t)),~~X(0)=x,$$
and
$$\Om(t) = \tau(t)\ul \Om,  \quad \zeta(t) = \ul\zeta\circ\kappa(t)\in 
\ol{\sR(\Om(t))}^w ,$$
 where  $\kappa(t)$ denotes the inverse of $\tau(t)$. 
We denote by $\overline \psi(t)$ the solution of 
\eqref{2.1a} to \eqref{special xi}
corresponding to $\Omega(t)$ and $\zeta(t)$,
and we set
$$\xi(t)=\overline \psi(t)|_{\mathscr S(t)}$$
$$\overline \psi_0(t)=\overline \psi(t)-\lambda_{1,\ul\Om,\ul\zeta}x_2$$
$$\xi_0(t)=\overline \psi_0(t)|_{\mathscr S(t)}$$
$$\Psi_0(t)=\overline \psi_0(t)\circ \tau(t)$$
$$\Gamma(t)=[D\kappa(t)\circ \tau(t)]^T=[(D\tau(t))^{-1}]^{T}$$
($\Gamma(t)$ at $x$ is the transpose of the spatial derivative of
$\kappa$ evaluated at $\tau(t)(x)$).

Note that $\overline \psi(0)= \underline \psi$
and $\overline \psi_0(0)= \psi_0$.
Moreover the dependence of $\ol \psi(t)\circ  \tau(t)\in H^1_{per}(\ol\Om)$ 
with respect to 
$t$ is smooth,  because 
$C(\Om(t),1,0)$,
$\lambda_{1,\Om(t),\zeta(t)}$
and $\lambda_{2,\Om(t),\zeta(t)}$
are smooth in $t$,
as can be checked with the help of the formulae following
\eqref{eq: def best xi} and by arguing in the fixed domain
$\ul\Om$ (via the map $\tau(t)$) as in
\cite[after (1.14)]{BuTo}.
Then the map $t\rightarrow \sL(\Om(t),\xi(t),\zeta(t))$ reaches its minimum at $t=0$
and therefore its derivative vanishes at $t=0$. Let us compute
the derivative of the term corresponding to the kinetic energy.

First note that 
$$C(\Om(t), \xi_0(t),\zeta(t))=\mu-\lambda_{1,\ul\Om,\ul\zeta}P,~~
I(\Om(t), \xi_0(t),\zeta(t))=\nu-\lambda_{1,\ul\Om,\ul\zeta}PQ,$$ 
$$\text{det}\, D \tau(t)=1,~~
\text{det}\, D\kappa(t)=1$$
and
$$\int_{\Om(t)}\nabla \overline \psi_0(t)\cdot\nabla(\psi_0\circ\kappa(t))dx
=C(\Om(t),\xi_0(t),\zeta(t))\lambda_{2,\ul\Om,\ul\zeta}
+\int_{\Omega(t)}\zeta(t)(\psi_0\circ\kappa(t))dx$$
because $\psi_0\circ\kappa(t)|_{\{x_2=0\}}=0$,
 $\psi_0\circ\kappa(t)|_{\mathscr S(t)}=\lambda_{2,\ul \Om,\ul \zeta}$
and $\Delta\overline \psi_0(t)=-\zeta(t)$.
Hence
\begin{eqnarray*}
&&\int_{\ul \Om}\Gamma(t)\nabla \Psi_0(t)\cdot \Gamma(t)\nabla \psi_0 dx
=\int_{\Om(t)}\nabla(\Psi_0(t)\circ \kappa(t))\cdot \nabla (\psi_0
\circ\kappa(t)) dx
\\&&=
\int_{\Om(t)}\nabla \overline \psi_0(t)\cdot\nabla(\psi_0\circ\kappa(t))dx
=C(\Om(t),\xi_0(t),\zeta(t))\lambda_{2,\ul\Om,\ul\zeta}
+\int_{\Omega(t)}\zeta(t)(\psi_0\circ\kappa(t))dx
\\&&=
(\mu-\lambda_{1,\ul\Om,\ul\zeta}P)\lambda_{2,\ul\Om,\ul\zeta}
 +\int_{\ul\Om}\ul\zeta\,   \psi_0\,dx  .
\end{eqnarray*}
By differentiating with respect to $t$ at $t=0$ in the equation
$$\int_{\ul\Om} \Gamma(t)\nabla \Psi_0(t)\cdot \Gamma(t) \nabla 
\psi_0\,dx  =
(\mu-\lambda_{1,\ul\Om,\ul\zeta}P)\lambda_{2,\ul\Om,\ul\zeta}
 +\int_{\ul\Om}\ul\zeta\,   \psi_0\,dx, $$
we get
\begin{equation}
\int_{\ul \Om} \nabla \dot \Psi_0(0)\cdot \nabla\psi_0
 \, d\x d\y  +2
  \int_{\ul \Om} \nabla  \psi_0  \cdot \dot\Gamma(0) 
\nabla \psi_0 \, dx=0  .\label{lineard}
\end{equation}

Let
$$
K(t)=\half\int_{\Om(t)} |\nabla \ol\psi(t)|^2\,dx  
=\half\int_{\Om(t)} |\nabla \overline \psi_0(t)|^2\,dx  
+\nu\lambda_{1,\ul\Om,\ul\zeta}+\frac 1 2 \lambda_{1,\ul\Om,\ul\zeta}^2PQ .
$$
Then
\begin{align*}\dot K(0)& = \frac{d~}{dt} \left.\left(\half\int_{\Om(t)} |\nabla \ol\psi_0(t)|^2\,dx  \right)\right|_{t=0}
\\& = \frac{d~}{dt} \left.\left(\half\int_{\ul \Om} |\Gamma(t)\nabla \Psi_0(t)|^2\,dx  \right)\right|_{t=0}
\\&= \int_{\ul\Om} \nabla \psi_0 \cdot \dot\Gamma(0)\nabla \psi_0 dx
  +   \int_{\ul\Om}  \nabla \psi_0 \cdot \nabla \dot \Psi_0(0) dx
\\&= -\int_{\ul\Om}  \nabla \psi_0\cdot \dot\Gamma(0)\nabla \psi_0\,  dx 
 \end{align*}
by \eqref{lineard}.
Now 
$$\Gamma(t)(\x,\y) = (D\t(t)[\x,\y]^T)^{-1} = I-t D\omega[\x,\y]^T +o(t) \text{ as } t \to 0,$$
$\dot\Gamma(0)= - D\om^T$ and
$$ \dot K(0)= \int_{\ul\Om}  \nabla \psi_0
\cdot D\om\nabla \psi_0\,  dx .$$
The end of the proof is as in \cite{BuTo}.

\end{proof}

\section{Minimization}
In what follows, the H\"{o}lder exponent $\gamma$ is equal to $1/4$,
so that in particular $H^2_{loc}(\RR)\subset C^{1,\gamma}(\RR)$.

Let $\sP$ be the set of all injective $H^2_{loc}$-functions 
$p:\RR \to \RR\times (0,\infty)$ such that
$p(x+P) = p(x) + (P,0)$ for all $x$,
 $p_1(0)=0$ and  $|p'|$ is constant. The length $\ell_p$ of $p([0,P])$
is equal to
$\ell_p = \int_{0}^{P} |p'(x)|\,dx$ and thus $|p'(x)|=\ell_p/P$ everywhere. 
We shall use the notation
$$\mathscr S_p =p(\RR) \text{ and }\sS_p=p((0,P)).$$ 
For $p\in \sP$,
we shall write $p\in \sP_Q$ if there exists $\mathit \Om\in \kO$ such that
the corresponding upper boundary 
$\mathscr S$ satisfies $\mathscr S=\mathscr S_p$. We shall then write
$$\mathit \Om_p=\mathit \Om
\text{ and } \Om_p=((0,P)\times\RR)\cap \mathit \Om.$$
We supplement the definition of $\sL$ (see \eqref{eq: definition of sL} 
and \eqref{eq: definition of sE}) by setting
 $$\sL(\Om_p,\xi,\zeta)=+\infty \text { for }p \not \in  \sP_Q.$$ 
In particular
 $\sL(\Om_p,\xi,\zeta)=+\infty $ if $p\in \sP$ is such that the area
of $\Om_p$ is different from $PQ$.

Also, if $\sP_Q\ni p_i \wk p \in \sP_Q$ in  $H^{2}_{\rm loc}(\RR,\RR^2)$, then
 \begin{equation*} \ell_p = \lim_{i \to \infty} \ell_{p_i} \text{ and } \int_0^{P}|p''|^2 ds \leq \liminf_{i \to \infty} \int_0^{P}|p_i''|^2 ds .  \end{equation*}

The next lemma leads to  an explicit criterion for the free surface 
to remain away from the bottom.
\begin{lemma} For any $p \in \sP_Q$,
\begin{equation} Q  \leq   \min p_2(\RR) +\frac{P}{2\pi} 
a\left(\frac{2\pi}{P}\ell_p\right),
\label{eq: above bottom}
\end{equation}
where $2\pi a(\ell)$ (when  $\ell > 2\pi$)
 is the area enclosed between a circular  arc of  length $ \ell$ and a chord of length $2\pi$, and thus
$$\frac{P^2}{2\pi}a\left(\frac{2\pi}{P}\ell\right)$$
 is the area enclosed between a circular  arc of  length $ \ell$ and a chord of length $P$.

Moreover
\begin{equation}
\label{eq: gravity energy bounded from below}
\int_{\Om_p} \y d\x d\y \geq P Q^2/2 .
\end{equation}
\end{lemma}
\begin{proof} See \cite{BuTo}.
\end{proof}

As a consequence, if $T+E>0$, then $\sL\geq gPQ^2/2$ 
with equality exactly when $\Omega_p=\Omega_Q$ and the fluid
is at rest
(see \eqref{eq: definition of sL} and \eqref{eq: definition of sE}).

The following lemma,  taken from \cite{BuTo},
provides an explicit way of ensuring that the free surface is
without double points, namely, it is sufficient to check that
inequality \eqref{eq: double point} below does not hold.
\begin{lemma}\label{lemma: injectivity}
Suppose that $p\in H^2_{loc}(\RR,\RR^2)$ 
is not injective and satisfies
$p(x+P) = p(x) + (P,0)$ for all $x$.
Then $p(\RR)$  contains
a closed loop with arc length no greater than $\ell_{p} -P$
(see \cite{HoTo}).
Let
$$p'(x) =|p'(x)| (\cos \vt(s),\sin \vt(s))
=P^{-1}\ell_p (\cos \vt(s),\sin \vt(s)),$$ 
where  $s=x \ell_p/P$ denotes arc length. 
Then, on the loop, the range of $\vt$ must exceed  $\pi$ and thus,
for some $0\leq s_2- s_1\leq\ell_{p}-P$,
\begin{align*}&\pi\leq|\vt(s_2) -\vt(s_1)| 
\leq P\ell_p^{-1}\int_{s_1P/\ell_p}^{s_2P/\ell_p}|p'' (x)|dx
\\&\leq  P\ell_p^{-1}\sqrt{P\ell_p^{-1}|s_2-s_1|}\|p''\|_{L^2(0,P)}
 \leq  \sqrt{\ell_{p} -P} \left(\frac{P}{\ell_{p}}\right)^{3/2}
\|p''\|_{L^2(0,P)} ,
\end{align*} 
hence
\begin{equation}
\label{eq: double point} 
\pi \leq  \sqrt{\ell_{p} -P} \left(\frac{P}{\ell_{p}}\right)^{3/2}
\|p''\|_{L^2(0,P)} .
\end{equation}
\end{lemma}

Let
$$W=\{(\Omega,\xi,\zeta): ~
p\in \sP_Q,~
\Om=\Om_p\in \kO,~
 \xi\in  H^{1/2}_{per}(\mathscr S_p),~ 
 \zeta \in L^2(\Omega)\},$$
\begin{multline*}
V:=\{(\Omega,\xi,\zeta)\in W: ~
 \zeta \in \overline{\sR(\Omega)}^w,~
 C(\Omega,\xi,\zeta)=\mu,~
I(\Omega,\xi,\zeta)=\nu\}.
\end{multline*}

By \eqref{eq: definition of sL}, \eqref{eq: definition of sE} and
\eqref{eq: gravity energy bounded from below},
if $T>0$
there is a bounded subset of $(0,P)\times(0,\infty)$
that contains all domains $\Omega$ such that, for some $\xi$
and $\zeta$, 
$(\Omega,\xi,\zeta)\in W$ and
$\sL(\Omega,\xi,\zeta)<\inf_V\sL+1$; hence
\begin{equation}
\label{eq: bound on Omega}
\exists R>0~ \forall (\Omega,\xi,\zeta)\in W~~
\Big( \sL(\Omega,\xi,\zeta)<\inf_V\sL+1\Rightarrow
\overline{\Omega}\subset [0,P]\times [0,R)  \Big).
\end{equation}

Let
$$\sR=\{\zeta\in L^2((0,P)\times(0,R)): ~\zeta\text{ is a rearrangement of }\zeta_Q\}$$
and $\overline{\sR}^w$ be its weak closure in $L^2((0,P)\times(0,R))$.

Hypothesis 
(M2) in the following existence result is related to the
various inequalities arising in the two previous lemmata.

\begin{theorem}
\label{thm: existence minimizer}
Assume that
\begin{itemize}
\item[(M1)]
$V$  does not contain any $(\Om,\xi,\zeta)$ with $\Om=\Om_Q$,
\item[(M2)] there exist $\sL_0>\inf_V \sL$,
$T>0$, $\beta\geq 1$
and $E>0$ such that
\begin{equation}\label{eq: choice T beta}
\frac{P}{2\pi}a\left(
\frac{2\pi}{P}\left\{\frac{\sL_0-\frac g 2 PQ^2}{T}\right\}^{1/\beta}
+2\pi\right)<Q,
\end{equation}
and
\begin{equation}\label{eq: choice E}
(\sL_0-\frac g 2 PQ^2)
\left\{\frac{\sL_0-\frac g 2 PQ^2}{T}\right\}^{1/\beta}
< E\pi^2
\end{equation}
(see \eqref{eq: definition of sE} for the meaning of $T$, $\beta$ and $E$).
\end{itemize}
Then $\inf_V\sL$ is attained.
\end{theorem}

{\bf Remarks} 

1) If we allow $\sL_0=\inf_V \sL$ in (M2) or 
require $\sL_0=\inf_V\sL$,
we do not change the meaning of (M2); however
$\sL_0>\inf_V \sL$ will be used in the proof of
Theorem \ref{thm: minimization}.
Note that, by Lemma \ref{lemma: on lambda_i},
$V\neq \emptyset$.

2) Assumption 
(M1) holds 
if $\zeta_Q$ is essentially one-signed and not trivial,
and  $(\nu-Q\mu)\zeta_Q\leq 0$ a.e.
(or $\nu-Q\mu\neq 0$ if $\zeta_Q$ vanishes a.e.).
See the paragraph ``A way of avoiding parallel flows'' in the introduction.

3) To see that all assumptions can be fulfilled, choose 
any $T>0$, $\beta\geq 1$ and $E>0$, and then
choose $\sL_0>\frac g 2 PQ^2$ near enough to 
$\frac g 2 PQ^2$ so that \eqref{eq: choice T beta} and \eqref{eq: choice E}
hold (this is possible because
$a(s)\rightarrow 0$ as $s\rightarrow 2\pi$ from the right).
Choose $p\in\sP_Q$ near enough to $(0,Q)$ in $H^2_{loc}$ and such that 
$\Om_p\neq \Om_Q$.
We know that 
$$I(\Om_p,1,0)-QC(\Om_p,1,0)=P-QC(\Om_p,1,0)<0$$
(see \eqref{eq: non-degeneracy}).
Choose
$\zeta_Q$ essentially non-negative and small enough in $L^2(\Om_Q)$,
and $\zeta\in\sR(\Om_p)$ such that
$I(\Om_p,1,\zeta)-QC(\Om_p,1,\zeta)<0$.
For  $\epsilon>0$, we have
$I(\Om_p,\epsilon,\epsilon\zeta)-QC(\Om_p,\epsilon,\epsilon\zeta)<0$.
 We then set
$\mu_\epsilon=C(\Om_p,\epsilon,\epsilon\zeta)$ and $\nu_\epsilon=I(\Om_p,
\epsilon,\epsilon\zeta)$. 
For $V_\epsilon$ corresponding  $\epsilon \zeta_Q$, $\mu_\epsilon$ and 
$\nu_\epsilon$, we get that
$(\Om_p,\epsilon,\epsilon\zeta)\in V_{\epsilon}$ and,
if $p-(0,Q)\in H^2_{loc}$ 
and $\epsilon$ are small enough,  that $\inf_{V_\epsilon}\sL<\sL_0$.

4) For the above choice of $V_\epsilon$, the minimizer  turns out to be
near $\Omega_Q$ (as $p-(0,Q)$ and $\epsilon>0$ above are chosen small enough).
However in order to check (M2)
for general $T>0$, $\beta\geq 1$, $E>0$
$\zeta_Q$, $\mu$ and $\nu$,  it is enough to
exhibit explicitly an appropriate $(\Om,\xi,\zeta)\in V$.
The observation that
the free surface of $\Omega$ need not be a graph (but must not touch 
or intersect itself) was intended to help addressing this question,
for example by numerical simulations.

The previous theorem is an immediate consequence of the following one.
For convenience write $\ol\psi(p,\zeta,\widetilde \mu,\widetilde \nu)$ 
for the solution to \eqref{2.1a}-\eqref{special xi}
corresponding to
the domain $\Om_{ p}\neq \Om_Q$ (with $p\in\sP_Q$),
the vorticity function $\zeta$, circulation $\widetilde \mu$
and horizontal impulse $\widetilde \nu$,
 and write
$\ol\xi(p,\zeta,\widetilde \mu,\widetilde \nu)
=\ol\psi(p,\zeta,\widetilde \mu,\widetilde \nu)|_{\mathscr S_p}$.
Moreover we write $\ol\lambda_1(p,\zeta,\widetilde \mu,\widetilde \nu)$
and $\ol\lambda_2(p,\zeta,\widetilde \mu,\widetilde \nu)$
for the corresponding $\lambda_1$ and $\lambda_2$ given by Lemma
\ref{lemma: on lambda_i} applied to
$\Om_{ p}\neq \Om_Q$,
$\zeta$, $\widetilde \mu$ and
$\widetilde \nu$.
\begin{theorem}
\label{thm: minimization}
As in Theorem \ref{thm: existence minimizer},
assume (M1) and (M2).
For each $k\in \NN$, let
$p_k \in\sP_Q$  with $\Omega_{p_k}\neq \Omega_Q$, 
 $\zeta_k \in L^2(\Om_{p_k})\subset L^2((0,P)\times(0,\infty))$ 
and $\mu_k,\nu_k\in \RR$.
Suppose  that
 $$\text{dist}_{L^2((0,P)\times (0,\infty))}
\left(\zeta_k,\ol{\sR}^w\right)\rightarrow 0 ,$$ 
$$\lim_{k\rightarrow \infty} \mu_k=\mu,~~
\lim_{k\rightarrow \infty} \nu_k=\nu$$
and
\begin{multline}
\label{eq: liminf less}
\limsup_{k\rightarrow \infty}\sL(\Om_{p_k},\ol \xi(p_k,\zeta_k,
\mu_k,\nu_k),\zeta_k)=
\\
\limsup_{k\rightarrow \infty}
\Big\{
\frac 1 2 \int_{\Om_{p_k}}|\nabla \ol \psi(p_k,\zeta_k,\mu_k,\nu_k)|^2
\,dx
+ g \int_{\Om_{p_k}} x_2\,dx + T (\ell_{p_k}-P)^\beta 
\\+E\left(\frac{P}{\ell_{p_k}}\right)^3 \int_{0}^{P}|p_k''(x)|^2 dx
\Big\}\leq \inf_V \sL.
\end{multline}
In particular these hypotheses hold true if
$\{(\Om_{p_k},\ol\xi(p_k,\zeta_k,\mu_k,\nu_k),
\zeta_k)\}_{k\in\NN}$ is a minimizing sequence in $V$ of $\sL$ 
(and thus $\mu_k=\mu$ and $\nu_k=\nu$ for all $k$).

Then there is a sequence $\{k_j\}\subset \NN$ such that
$\{p_{k_j}\}$ converges weakly in $H^2_{per}$ to some $p\in \sP_Q$ and
$\{\zeta_{k_j}\}$ seen in $L^2((0,P)\times(0,\infty))$ 
converges weakly to some $\zeta\in L^2(\Om_p)$.
Moreover $L^2((0,P)\times(0,\infty))$
can be seen as a subspace of the dual space
$(H^1((0,P)\times(0,\infty))\big)'$ of $H^1((0,P)\times(0,\infty))$ and
\begin{equation}
\label{eq: limit in H1'}
\zeta_{k_j}\rightarrow \zeta \text{ strongly in } 
\big(H^1((0,P)\times(0,\infty))\big)'.
\end{equation}

Since $\Om_p\in\kO$,  there
exists a $C^{1,\gamma}$-map $F:\RR^2\rightarrow \RR^2$ 
satisfying the following
properties:
\begin{itemize}
\item $F$ restricted to $\RR\times [0,Q]$ is a diffeomorphism
from $\RR\times [0,Q]$ onto $\overline{\mathit \Om_p}$,
\item $F(x_1,0)=(x_1,0)$ for all $x_1\in\RR$,
\item $F$ 
 restricted to $\RR\times\{Q\}$ is a homeomorphism from
$\RR\times \{Q\}$ onto $\mathscr S_p$,
\item $F(x_1+P,x_2)
=(F_{1}(x_1,x_2)+P,F_{2}(x_1,x_2))$ for all $x=(x_1,x_2)\in \RR\times[0,Q]$.
\end{itemize}
In the same way as for $F$, we introduce $F_{j}:\RR^2\rightarrow \RR^2$
such that, restricted to $\RR\times [0,Q]$, it is a diffeomorphism
from $\RR\times [0,Q]$ onto $\overline{\mathit \Om_{p_{k_j}}}$.

Then this can be done in such a way that
\begin{equation}\label{eq: lim unif 1}
||F_{j}-F||_{C^1(\sU)}\rightarrow 0\text{ for some open set $\sU$
containing $\ol{\mathit \Om_Q}$},
\end{equation}
\begin{equation}\label{eq: lim unif 2}
\ol\lambda_{1}(p_{k_j},\zeta_{k_j},\mu_{k_j},\nu_{k_j})\rightarrow 
\ol\lambda_{1}(p,\zeta,\mu,\nu),~~
\ol\lambda_{2}(p_{k_j},\zeta_{k_j},\mu_{k_j},\nu_{k_j})\rightarrow 
\ol\lambda_{2}(p,\zeta,\mu,\nu)
\end{equation}
and
\begin{equation}\label{eq: lim unif 3}
||\ol \psi(p_{k_j},\zeta_{k_j},\mu_{k_j},\nu_{k_j})-\ol\psi(p,\eta,
\mu,\nu)||_{H^1_{per}((0,P)\times (0,R))}
\rightarrow 0,
\end{equation}
where $R$ is large enough so that the closures of $\Om_p$ and all 
$\Om_{p_k}$ 
are subsets of $[0,P]\times [0,R)$ 
(see \eqref{eq: bound on Omega})
and where $\ol\psi(p,\zeta,\mu,\nu)$ 
and all $\ol \psi(p_k,\zeta_k,\mu_k,\nu_k)$ have been
extended in $(0,P)\times (0,R)$ by 
$\ol\lambda_{1}(p,\zeta,\mu,\nu) x_2+
\ol\lambda_{2}(p,\zeta,\mu,\nu)$ and
$\ol\lambda_{1}(p_k,\zeta_k,\mu_k,\nu_k) x_2+\ol\lambda_{2}
(p_k,\zeta_k,\mu_k,\nu_k)$.

Finally $(\Omega_p,\ol\xi(p,\zeta,\mu,\nu),\zeta)\in V$,
$\sL(\Omega_p,\ol\xi(p,\zeta,\mu,\nu),\zeta)=\inf_V \sL$,
the limsup in \eqref{eq: liminf less} is a limit:
\begin{equation}
\label{eq: limsup is limit}
\lim_{k\rightarrow \infty}\sL(\Om_{p_k},\ol \xi(p_k,\zeta_k,\mu_k,\nu_k),\zeta_k)=\inf_V \sL
\end{equation}
and
\begin{equation}\label{eq: strong convergence}
p_{k_j}\rightarrow p ~\text{ strongly in }~H^2_{per}~.
\end{equation}

\end{theorem}

\begin{proof}
Let $p_k \in \sP_Q$,
 $\zeta_k\in  L^2(\Om_{p_k}) $ and $\mu_k,\nu_k\in\RR$ be such that
 $$\text{dist}_{L^2((0,P)\times (0,\infty))}
\left(\zeta_k,\ol{\sR}^w\right)\rightarrow 0, $$ 
$$\lim_{k\rightarrow \infty} \mu_k=\mu,
~~\lim_{k\rightarrow \infty} \nu_k=\nu$$
and
\begin{multline*}
\limsup_{k\rightarrow \infty}\Big\{
\frac 1 2 \int_{\Om_{p_k}}|\nabla \ol 
\psi(p_k,\zeta_k,\mu_k,\nu_k)|^2\,dx
+ g \int_{\Om_{p_k}} x_2\,dx + T (\ell_{p_k}-P)^\beta 
\\+E\left(\frac{P}{\ell_{p_k}}\right)^3 \int_{0}^{P}|p_k''(x)|^2 dx
\Big\}\leq
\inf_V \sL  ,
\end{multline*}

For simplicity, we set
$$\ol\psi_k=\ol\psi(p_k,\zeta_k,\mu_k,\nu_k),~~
\ol\xi_k=\ol\xi(p_k,\zeta_k,\mu_k,\nu_k)$$
$$\ol\lambda_{1,k}=\ol\lambda_1(p_k,\zeta_k,\mu_k,\nu_k)
~\text{ and }~
\ol\lambda_{2,k}=\ol\lambda_2(p_k,\zeta_k,\mu_k,\nu_k)$$
(remember that  we write $\ol\lambda_{1}(p_k,\zeta_k,\mu_k,\nu_k)$
and $\ol\lambda_2(p_k,\zeta_k,\mu_k,\nu_k)$
for the corresponding $\lambda_1$ and $\lambda_2$ given by Lemma
\ref{lemma: on lambda_i} applied to
$\Om_{ p_k}\neq \Om_Q$).

We get, for all $k\in\NN$ large enough,
\begin{equation}
\label{eq: energy inequality}
\begin{array}{l}
\frac 1 2 \int_{\Om_{p_k}}|\nabla \ol \psi_k|^2\,dx
+ \frac g 2 PQ^2  
\\+ T (\ell_{p_k}-P)^\beta 
+E\left(\frac{P}{\ell_{p_k}}\right)^3 \int_{0}^{P}|p_k''(x)|^2 dx
\stackrel{\eqref{eq: gravity energy bounded from below}}{\leq}
\sL(\Om_{p_k},\ol\xi_k,\zeta_k)\leq \sL_0~,
\end{array}\end{equation}
\begin{equation}
\label{eq: intermediate}
\ell_{p_k}-P\stackrel{\eqref{eq: energy inequality}}
{\leq} \left\{\frac{\sL_0-\frac g 2 PQ^2}{T}\right\}^{1/\beta}~,
\end{equation}
\begin{equation}\label{eq: bounded from below}
\frac{P}{2\pi}a\left(\frac{2\pi}{P}\ell_{p_k}\right)\stackrel{
\eqref{eq: choice T beta}}{<}Q
,~~\min p_{k,2}(\RR)\stackrel{\eqref{eq: above bottom}}{>}0,
\end{equation}
\begin{equation}\label{eq: uniform estimate}
(\ell_{p_k}-P)\left(\frac{P}{\ell_{p_k}}\right)^3||p_k''||^2_{L^2(0,P)}
\\\stackrel{\eqref{eq: intermediate},\eqref{eq: choice E}
}{<}
\pi^2\frac{E \left(\frac{P}{\ell_{p_k}}\right)^3||p_k''||^2_{L^2(0,P)}}
{\sL_0-\frac g 2 PQ^2}\stackrel{
\eqref{eq: energy inequality}
}{\leq} \pi^2
\end{equation}
uniformly in $k$ large enough.
Observe that
$\{ p_k \}$ is bounded in $H^2_{per}$
because 
$\sL(\Om_{p_{k}},\ol\xi_{k} ,\zeta_{k})$ is  bounded and $T,E>0$.
So there is a sequence $\{k_j\}\subset \NN$ such that
$\{p_{k_j}\}$ converges weakly in $H^2_{per}$ to some $p$ and
$\{\zeta_{k_j}\}$ seen in $L^2((0,P)\times(0,\infty))$ 
converges weakly to some $\zeta\in L^2(\Om_p)$.

Remember the constant $R>0$ introduced in 
\eqref{eq: bound on Omega}. As in fact
$\{\zeta_{k_j}\}\subset L^2((0,P)\times(0,R))$ and as
 the inclusion map
$L^2((0,P)\times(0,R))\subset \big(H^1((0,P)\times(0,R))\big)'$
is compact, we get \eqref{eq: limit in H1'}.

By lemma \ref{lemma: injectivity} and \eqref{eq: uniform estimate},
$p$ is injective and, by
\eqref{eq: bounded from below}, $p(\RR)\subset \RR\times (0,\infty)$.
Hence $p\in \sP_Q$
(see \cite{BuTo}).

Let $F$ be as in the statement.  Then $F$ restricted to some
open neighbourhood $\sU$ of
$\RR\times[0,Q]$ is still a diffeomorphism onto the open set $F(\sU)$
containing $\ol{\mathit \Om_p}$. 
As a consequence, for large enough
$j$, $\mathit \Om_{p_{k_j}}\subset F(\sU)$ and $F^{-1}(\mathscr S_{p_{k_j}})$
is the graph of a map $x_1\rightarrow H_j(x_1)$
that is $C^1$-close to the constant map $x_1\rightarrow x_2= Q$.
Define 
$$G_j(x_1,x_2)=(x_1,x_2\, H_j(x_1)/Q) \text{ and } F_j=(F_{j1},F_{j2})
= F \circ G_j
\text{ for all }j.$$
Then 
\eqref{eq: lim unif 1} holds.
Extend $\ol \psi_k$ 
on $(0,P)\times (0,R)$, as in the statement.
Observe that
$$\sup_{j\in\NN}\int_{\Om_{p_{k_j}}}
|\nabla \ol \psi_{k_j}|^2dx<\infty$$  because 
$\sL(\Om_{p_{k_j}},\ol\xi_{k_j} ,\zeta_{k_j})$ is finite.
Hence we get successively
$$\sup_{j\in\NN}\int_{\Om_{Q}}
|\nabla \left(\ol \psi_{k_j}
\circ F_j\right)|^2dx<\infty,$$
$$\sup_{j\in\NN}||
\ol \psi_{k_j}\circ F_j||_{H^1(\Om_Q)}<\infty$$
by Poincar\'e's inequality,
$$\sup_{j\in\NN}\int_0^P
\Big|\ol \psi_{k_j}\circ F_j\big|_{x_2=Q}\Big|^2dx_1<\infty,$$
or, equivalently,
$$\sup_{j\in\NN}\int_0^P|\ol\lambda_{1,k_{j}}
F_{j2}(x_1,Q)
+\ol\lambda_{2,k_{j}}|^2
dx_1<\infty.$$

Suppose first that 
$\{\ol\lambda_{1,k_j}\}$ is unbounded.
Taking a subsequence if necessary,
$F_{j2}(\cdot,Q)+(\ol\lambda_{2,k_j}/
\ol\lambda_{1,k_j})$  
would converge to $0$ in $L^2(0,P)$, and therefore
$F_{2}(x_1,Q)=Q$ for all $x_1\in (0,P)$
(this follows from \eqref{eq: lim unif 1}).
Hence $\Omega_p=\Omega_Q$.

Let $\widetilde Q\in (Q/2,Q)$.
From the Poincar\'e inequality, it follows that 
the sequence $\{\ol\psi_{k_j}\}$ 
seen in $H^1_{per}((0,P)\times(0,\widetilde Q))$ 
is bounded too
 and therefore,
up to a subsequence, it converges weakly to some 
$\psi_{\widetilde Q}\in H^1_{per}((0,P)\times(0,\widetilde Q))$. 
Moreover this can be achieved in such a way that there exists 
$\psi\in H^1_{per}(\Omega_Q)$ independent of
$\widetilde Q$ such that $\psi_{\widetilde Q}$
and $\psi$ are equal on
$(0,P)\times(0,\widetilde Q)$. 
Also, up to a subsequence, the sequence $\{\zeta_{k_j}\}$ 
seen in  $L^2((0,P)\times(0,R))$
converges weakly to some $\zeta$ that belongs in fact to $L^2(\Om_{p})
=L^2(\Om_{Q})$, that is, $\zeta$ vanishes almost everywhere outside
$\Om_Q$. Moreover $\zeta\in \overline{\sR}^w$
because
$$\text{dist}_{L^2((0,P)\times (0,\infty))}\left(\zeta_{k_j},\ol{\sR}^w\right)\rightarrow 0.$$
In fact $\zeta$ 
even belongs to the convex set 
$\overline{\sR(\Om_Q)}^w$,
as it can be seen from the characterisation of 
$\overline{\sR(\Om)}^w$ for any open bounded set $\Om$ of measure $m>0$
in terms of decreasing rearrangements on $[0,m]$. See e.g. Lemma 2.2 in \cite{BuMc}.
\footnote{
\label{footnote}
Indeed let $g_1:(0,PQ)\rightarrow \RR$ 
be the right-continuous and
decreasing rearrangement of $\zeta\in L^2(\Omega_Q)$. 
If $\zeta$ is seen in $L^2((0,P)\times(0,R))$ instead, we can also consider
its right-continuous and decreasing rearrangement $g_2 :(0,PR)\rightarrow \RR$.

Note that $g_2$ vanishes on an interval $Z_\zeta$ of length at least $PR-PQ$.
Moreover the graph of $g_1$ is obtained from the one of $g_2$ by deleting from $Z_\zeta$ an 
interval of length $PR-PQ$ and shifting to the left the part of the graph of $g_2$ that is to the right of $Z_\zeta$.

We note by $G_1$ and $G_2$ the rearrangements corresponding to $\zeta_Q$.

With the partial ordering $\prec$ of Burton-McLeod (see their lemma 2.2), we get 
successively
$\zeta\in \overline{\sR}^w$,  $g_2\prec G_2$, $g_1\prec G_1$ and therefore $\zeta\in \overline{\sR(\Omega_Q)}^w$.
}

Let $\xi=\psi|_{(0,P)\times\{Q\}}$. Then,
in a weak sense, $-\Delta\psi=\zeta$ on
$\Omega_Q$, $\psi(\cdot,0)=0$  and $\psi(\cdot,Q)=\xi$.

By choosing 
$\widehat \psi\in H^1_{per}((0,P)\times (0,R))$ such that
$\widehat \psi$ restricted to $\{x_2=0\}$ vanishes and such that
$\widehat\psi=1$ on $(0,P)\times(Q/3,R)$,
we get that
\begin{multline}
\label{eq: computation I}
\mu=\lim_{j\rightarrow \infty }\mu_{k_j}
=\lim_{j\rightarrow \infty }C(\Om_{p_{k_j}},\ol \xi_{k_j} ,\zeta_{k_j})
=\lim_{j\rightarrow \infty }\int_{\Om_{p_{k_j}}}\{\nabla\ol\psi_{k_j}
\cdot \nabla \widehat \psi-\zeta_{k_j}\widehat\psi\} dx
\\=\lim_{j\rightarrow \infty }
\int_{(0,P)\times(0,Q/2)}  \nabla \ol\psi_{k_j}\cdot \nabla \widehat \psi  dx
-\lim_{j\rightarrow \infty }\int_{(0,P)\times(0,R)}\zeta_{k_j}\widehat\psi dx
\\=\int_{\Om_{Q}}\{\nabla\psi\cdot \nabla \widehat \psi-\zeta\widehat\psi\} dx
=C(\Om_{Q},\xi,\zeta)
\end{multline}
and 
\begin{multline}
\label{eq: computation II}
\nu=\lim_{j\rightarrow \infty }\nu_{k_j}
=\lim_{j\rightarrow \infty }\int_{\Om_{p_{k_j}}}\nabla\ol\psi_{k_j}
\cdot \nabla x_2\,  dx
\\
=\lim_{\widetilde Q\rightarrow Q^-}\lim_{j\rightarrow \infty }
\int_{(0,P)\times(0,\widetilde Q)}\nabla\ol\psi_{k_j}
\cdot \nabla x_2\,  dx
=\int_{\Om_{Q}}\nabla\psi\cdot \nabla x_2\,  dx
=I(\Om_{Q},\xi,\zeta).
\end{multline}

Hence $(\Om_Q,\xi,\zeta) \in V$, which contradicts 
(M1).
As a consequence
 $\{\ol\lambda_{1,k_j}\}$ is bounded.
We now apply some of the above arguments again.

From the Poincar\'e inequality, it follows that 
the sequence $\{\ol\psi_{k_j}\}$ seen now 
in $H^1_{per}((0,P)\times(0,R))$ is bounded
 and therefore,
up to a subsequence, it converges weakly to some $\psi\in H^1_{per}((0,P)\times
(0,R))$.
In particular it follows that
$\{\ol\lambda_{2,k_j}\}$ is bounded.
Again, up to a subsequence, the sequence $\{\zeta_{k_j}\}$ 
seen in  $L^2((0,P)\times(0,R))$
converges weakly to some $\zeta$ that belongs 
to $\overline{\sR(\Om_p)}^w$.

By choosing again
 $\widehat \psi\in H^1_{per}((0,P)\times (0,R))$ such that
$\widehat \psi$ restricted to $\{x_2=0\}$ vanishes and such that
$\widehat\psi=1$ on some open set containing $\mathscr S_p$ and
 all $\mathscr S_{p_{k_j}}$, we get that
\begin{multline*}
\mu=\lim_{j\rightarrow \infty }C(\Om_{p_{k_j}},\ol \xi_{k_j} 
,\zeta_{k_j})
=\lim_{j\rightarrow \infty }\int_{\Om_{p_{k_j}}}\{\nabla\ol\psi_{k_j}
\cdot \nabla \widehat \psi-\zeta_{k_j}\widehat\psi\} dx
\\=\int_{\Om_{p}}\{\nabla\psi\cdot \nabla \widehat \psi-\zeta\widehat\psi\} dx
=C(\Om_{p},\psi|_{\mathscr S},\zeta)
\end{multline*}
and 
\begin{multline*}
\nu=
\lim_{j\rightarrow \infty }\int_{\Om_{p_{k_j}}}\nabla\ol\psi_{k_j}
\cdot \nabla x_2\,  dx
=\lim_{j\rightarrow \infty }\int_{\Om_{p}}\nabla\ol\psi_{k_j}
\cdot \nabla x_2\,  dx
\\=\int_{\Om_{p}}\nabla\psi\cdot \nabla x_2\,  dx
=I(\Om_{p},\psi|_{\mathscr S},\zeta).
\end{multline*}
By convexity, for all $\widetilde Q>Q$ and
 $q\in \sP_{\widetilde Q}$ such that $\Om_p\subset\Om_q\subset(0,P)
\times (0,R)$
and $\mathscr S_q\cap  \mathscr S_p=\emptyset$, we have
\begin{multline*}
\int_{\Om_q}|\nabla \psi|^2dx\leq 
\liminf_{j\rightarrow \infty}\int_{\Om_{q}}
|\nabla \ol\psi_{k_j}|^2dx
\\\leq \liminf_{j\rightarrow \infty}\Big(\int_{\Om_{p_{k_j}}}
|\nabla \ol\psi_{k_j}|^2dx
+\text{Const}\,
\text{meas}(\Om_q\backslash\Om_{p_{k_j}})\Big)
\end{multline*}
(because the sequence $\{\ol\lambda_{1,_{k_j}}\}$ is bounded)
and therefore
$$\int_{\Om_p}|\nabla \psi|^2dx\leq \lim_{j\rightarrow \infty}
\int_{\Om_{p_{k_j}}}|\nabla \ol\psi_{k_j}|^2dx.$$
It follows that $\sL(\Om_p,\psi|_{\mathscr S_p},\zeta)\leq\inf_V \sL$ and
$(\Om_p,\psi|_{\mathscr S_p},\zeta)\in V$ so
$\sL(\Om_p,\psi|_{\mathscr S_p},\zeta)=\inf_V \sL$
and $\psi=\ol\psi(p,\zeta,\mu,\nu)$. Hence
\eqref{eq: lim unif 2} holds and 
\begin{equation}\label{eq: equal integrals}
\int_{\Om_p}|\nabla \psi|^2dx= \lim_{j\rightarrow \infty}
\int_{\Om_{p_{k_j}}}|\nabla \ol\psi_{k_j}|^2dx.
\end{equation}
By \eqref{eq: lim unif 2}, for all
$\widetilde Q>Q$ and  $q\in \sP_{\widetilde Q}$ such that
 $\Om_p\subset \Om_{q}\subset (0,P)\times (0,R)$ and
$\mathscr S_p\cap \mathscr S_{q}=\emptyset$,
we have
$$\int_{((0,P)\times(0,R))\backslash\Om_{q}}|\nabla \psi|^2dx=
 \lim_{j\rightarrow \infty}
\int_{((0,P)\times (0,R))\backslash\Om_{q}}|\nabla \ol\psi_{k_j}|^2dx.$$

Hence
$$\int_{(0,P)\times(0,R)}|\nabla \psi|^2dx=
 \lim_{j\rightarrow \infty}
\int_{(0,P)\times (0,R)}|\nabla \ol\psi_{k_j}|^2dx$$
and \eqref{eq: lim unif 3} holds too.

Together with \eqref{eq: equal integrals}, the fact that
$$
\sL(\Om_{p_{k_j}},\ol\xi_{k_j},\zeta_{k_j})\rightarrow 
\sL(\Om_p,\psi|_{\mathscr S_p},\zeta)$$
implies that 
$$\int_0^P|p''_{k_j}|^2dx\rightarrow \int_0^P|p''|^2dx.$$
Hence $p_{k_j}\rightarrow p$ strongly in $H^2_{per}$ .
\end{proof}

\section{On stability}
\label{section: stability}

In this section, we assume that hypotheses
(M1) and (M2) in Theorem
\ref{thm: existence minimizer} hold true.
Moreover the H\"{o}lder exponent $\gamma$ is still equal to $1/4$.

For smooth flows, the evolutionary problem reads as follows (see
e.g. \cite{CoSt}).
Let $\psi(t,\cdot,\cdot)\in C^{\infty}_{per}(\mathit{\Om}(t))$ 
be the stream function at time $t$ on the domain
$\mathit{\Om}(t)\in \kO$, that is, the velocity field is given by
$u=(u_1,u_2)
=(\partial_{x_2}\psi,-\partial_{x_1}\psi)$ on $\mathit{\Om}(t)$.
The Euler equation for an inviscid flow becomes
$$\left\{\begin{array}{l}
\partial_{t}u_1+u_1\, \partial_{x_1}u_1+
u_2\,\partial_{x_2}u_1=-\partial_{x_1}\text{Pr}\\
\\
\partial_{t}u_2+u_1\, \partial_{x_1}u_2
+u_2\, \partial_{x_2}u_2=-\partial_{x_2}\text{Pr}-g
\end{array}\right.~~~\text{ on }\mathit{\Om}(t),
$$
where $\text{Pr}(t,x_1,x_2)$ is the pressure. The kinematic
boundary conditions are
$$\psi(t,x_1,0)=0 $$
on the bottom
and
$$\partial_tp-(\partial_{x_2}\psi,-\partial_{x_1}\psi)\in\text{span} 
\{\partial_s p\}$$
on the upper boundary $\mathscr{S}(t)$ of $\mathit{\Om}(t)$
that we assume of the form 
$$\mathscr{S}(t)=\{p(t,s)\in\RR\times(0,\infty):s\in\RR\}$$
with $p$ smooth such that $p(t,\cdot)\in \sP_Q$ for all $t\in\RR$.
The kinematic boundary condition on the top can also be written
$$\nabla \psi\cdot \partial_s p=\text{det}\, p' $$
where $\nabla$ is the gradient with respect to $(x_1,x_2)$
and $p'$ is the matrix of the first order partial derivatives
with respect to $t$ and $s$.
The dynamic boundary condition on the top reads
(compare with \eqref{bern''})
$$\text{Pr}=
 - T \beta \big(\ell(\sS(t))-P\big)^{\beta -1}\sigma
 + E\Big(2\sigma '' +\sigma^3\Big)+ 
\text{ function of $t$ only }
$$
on $\mathscr{ S}(t)$,
where $'$ denotes differentiation with respect to arc length along the surface
$\mathscr{S}(t)$,
$\sigma(t,x)$ is the curvature of the surface at $x\in \mathscr{S}(t)$ and
 $\ell(\sS(t))$ is the length of $\sS(t)$.

It is a standard result of classical hydrodynamics that the vorticity function
$\zeta=\partial_{x_1}u_2- \partial_{x_2}u_1=-\Delta \psi$
is convected by the flow,
where $\Delta$ is the Laplacian with respect to $(x_1,x_2)$.
Similarly
the circulation along the bottom
is preserved, thanks to the equation
$\partial_{t} u_1+(1/2)\partial_{x_1}(u_1^2)=-\partial_{x_1}
\text{Pr}$ available
at the bottom
(because $u_2=0$ there).  Hence
the circulation $C$ along one period of the free boundary
is preserved too.
These considerations have been the motivation for the variational problems
studied in this paper.

Let us begin our study of stability by defining a distance dist$_0$ 
between $(\Om_1,\xi_1, \zeta_1)$
and  $(\Om_2,\xi_2, \zeta_2)$ in the set
\begin{multline*}
W^*=\{(\Omega,\xi,\zeta): ~
p\in \sP_Q,~
\Om=\Om_p\in \kO\backslash\{\Om_Q\},~
\\
 \xi\in  H^{1/2}_{per}(\mathscr S_p),~ 
 \zeta \in L^2(\Omega)
\subset L^2((0,P)\times (0,\infty))\}
\end{multline*}
(in the definition of $W^*$, $\Omega=\Omega_Q$ is forbidden).
Let $R>0$ be given by \eqref{eq: bound on Omega}.
For $i\in\{1,2\}$, 
we write $\psi_{\Om_i,\xi_i,\zeta_i}$ 
for the solution to \eqref{2.1a}-\eqref{eq: H1per}
corresponding to the domain $\Om_i$, $\xi_i$ 
and the vorticity function $\zeta_i$.
Moreover we  write $\ol\psi(\Om_i,\zeta_i,\mu_i,\nu_i)$ 
for the solution to \eqref{2.1a}-\eqref{special xi}
corresponding to $\Om_{i}$, $\zeta_i$
$\mu_i=C(\Om_i,\xi_i,\zeta_i)$ and 
$\nu_i=I(\Om_i,\xi_i,\zeta_i)$
(see Lemma \ref{lemma: on lambda_i} for the existence of
such a solution), and then extended on 
$\{(0,P)\times (0,\infty)\}\backslash \Om_i$
in a way that is independent of $x_1$ and affine in $x_2$
(see equation \eqref{special xi} on the free boundary, $\xi$
in \eqref{2.1a}-\eqref{special xi} being now not given a priori).

If $\Om_1= \Om_2$ and $\overline\psi(\Om_1,\zeta_1,\mu_1,\nu_1)
=\overline\psi(\Om_2,\zeta_2,\mu_2,\nu_2)$, we set
\begin{multline*}
\text{dist}_0((\Om_1,\xi_1,\zeta_1),(\Om_2,\xi_2,\zeta_2))
=||\zeta_{1}-\zeta_2||_{(H^{1}((0,P)\times \RR))'} 
\\+||\nabla\psi_{\Om_1,\xi_1,\zeta_1}
-\nabla\psi_{\Om_2,\xi_2,\zeta_2}||_{L^2(\Om_1)}~,
\end{multline*}
(when actually $\zeta_1=\zeta_2$) and in all other cases write
\begin{eqnarray*}
&&\text{dist}_0((\Om_1,\xi_1,\zeta_1),(\Om_2,\xi_2,\zeta_2))
=\inf_{s\in[0,P]}||p_1(s+\cdot)-p_2||_{H^2_{per}}
\\&&+||\zeta_{1}-\zeta_2||_{(H^{1}((0,P)\times \RR))'} 
+||\nabla\psi_{\Om_1,\xi_1,\zeta_1}
-\nabla\overline \psi(\Om_1,\zeta_1,\mu_1,\nu_1)||
_{L^2(\Om_1)}
\\&&+||\nabla\psi_{\Om_2,\xi_2,\zeta_2}
-\nabla\overline \psi(\Om_2,\zeta_2,\mu_2,\nu_2)||
_{L^2(\Om_2)}
\\&&+||\nabla\overline \psi(\Om_1,\zeta_1,\mu_1,\nu_1)
-\nabla \overline \psi(\Om_2,\zeta_2,\mu_2,\nu_2)||
_{L^2((0,P)\times(0,R))}~
\end{eqnarray*}
for some parameterisations $p_1$ and $p_2$
of the free boundaries (that is, $p_1,p_2\in \sP_Q$, $\Om_1=\Om_{p_1}$ and
$\Om_2=\Om_{p_2}$). Observe that $p_1,p_2:\RR\rightarrow \RR^2$ 
are uniquely defined only up to translations in $s$.

Theorem \ref{thm: minimization}
implies that the set $D(\mu,\nu,\zeta_Q)$ of minimizers
of $\sL|_V$
endowed with the distance $\text{dist}_0$ is compact
(see \eqref{eq: limit in H1'} 
to \eqref{eq: strong convergence}).

\begin{lemma}
Let
$((\Om_n,\xi_n,\zeta_n):n\in\NN)\subset W$
be such that 
 $$\text{dist}_{L^2((0,P)\times (0,\infty))}
\left(\zeta_n,\ol{\sR}^w\right)\rightarrow 0 ,$$ 
$$C(\Om_n,\xi_n,\zeta_n)\rightarrow \mu,~~
I(\Om_n,\xi_n,\zeta_n)\rightarrow \nu,~~
\limsup_{n\rightarrow \infty}\sL(\Om_n,\xi_n,\zeta_n) \leq\inf_V\sL.$$

Then 
$\Om_n\neq \Om_Q$ for all $n$ sufficiently large, 
the distance  dist$_0$ of $(\Om_n,\xi_n,\zeta_n)$
to the set  $D(\mu,\nu,\zeta_Q)$ 
of minimizers converges to $0$ and
$\lim_{n\rightarrow \infty}\sL(\Om_n,\xi_n,\zeta_n) =\inf_V\sL.$
\end{lemma}
\begin{proof}
Let us first suppose that
$\Om_n\neq \Om_Q$ for all $n\in\NN$.
For each $n$, let $\mu_n=C(\Om_n,\xi_n,\zeta_n)$,
$\nu_n=I(\Om_n,\xi_n,\zeta_n)$ and
$p_n\in \sP_Q$ be such that $\Om_n=\Om_{p_n}$.

We write $\ol\psi_n$ 
for the solution to \eqref{2.1a}-\eqref{special xi}
corresponding to
the domain $\Om_{ n}\neq \Om_Q$,
the vorticity function $\zeta_n$, circulation $\mu_n$
and horizontal impulse $\nu_n$,
 and write
$\ol\xi_n$ for the trace of $\ol\psi_n$ to the upper boundary
of $\Omega_n$.
In particular
\begin{equation}\label{eq: clearly}
C(\Om_n,\ol\xi_n,\zeta_n)=\mu_n
~\text{ and } ~I(\Om_n,\ol\xi_n,\zeta_n)=\nu_n~.
\end{equation}
Moreover we write $\ol\lambda_{1n}$
and $\ol\lambda_{2n}$
for the corresponding $\lambda_1$ and $\lambda_2$ given by Lemma
\ref{lemma: on lambda_i} applied to
$\Om_{ n}$,
$\zeta_n$, $\mu_n$ and
$\nu_n$.

As
$$\sL(\Om_n,\ol\xi_n
,\zeta_n)\leq \sL(\Om_n,\xi_n,\zeta_n),$$
(see Proposition \ref{prop: fixed omega and zeta}),
we can apply Theorem \ref{thm: minimization} to the sequence
$\{(\Om_n,\ol\xi_n,\zeta_n)\}_{n\geq 1}$:
the distance  dist$_0$
of $(\Om_n,\ol\xi_n,\zeta_n)$
to the set  $D(\mu,\nu,\zeta_Q)$ 
of minimizers converges to $0$ (see 
\eqref{eq: limit in H1'} to 
\eqref{eq: strong convergence}).
We also have proved that there
is at least one minimizer.

This implies that the  distance dist$_0$ 
of $(\Om_n,\xi_{n},\zeta_n)$
to the set  $D(\mu,\nu,\zeta_Q)$ 
of minimizers converges to $0$. To see it,
we write $\psi_{\Om,\xi,\zeta}$ 
for the solution to \eqref{2.1a}-\eqref{eq: H1per}
corresponding to
the domain $\Om\neq \Om_Q$, $\xi$ 
and the vorticity function $\zeta$ (however $\xi$ is not assumed to
satisfy \eqref{special xi}).
We let $\widetilde \psi_n$ be, as in \eqref{eq: widetilde psi},
the harmonic function on $\mathit \Om_n$ that
vanishes on $\{x_2=0\}$, is $1$ on $\mathscr S_n$ and is $P$-periodic
in $x_1$.

Looking for a contradiction,
 assume that some subsequence, still denoted by
$\{(\Om_n,\xi_{n},\zeta_n)\}$, is such that
its  distance dist$_0$ 
to  $D(\mu,\nu,\zeta_Q)$  remains away from $0$.
Taking a further subsequence if needed, we may also assume that
$(\Om_n,\ol \xi_n,\zeta_n)$ tends to some $(\Om,\xi,\zeta)\in 
D(\mu,\nu,\zeta)$. 
We get
\begin{eqnarray*}
&&\int_{\Omega_{n}}|\nabla(\psi_{\Om_n,\xi_n,\zeta_n}
-\ol\psi_n)|^2dx
=\int_{\Omega_{n}}|\nabla \psi_{\Om_n,\xi_n,\zeta_n}|^2dx
-\int_{\Omega_{n}}|\nabla \ol \psi_n |^2dx
\\&&-2 \int_{\Omega_{n}}\nabla 
(\ol\psi_n-
\ol\lambda_{1,n}x_2- \ol\lambda_{2,n}\widetilde\psi_n
)\cdot\nabla(
\psi_{\Om_n,\xi_n,\zeta_n}-\ol\psi_{n})
\, dx
\\&&-2 \int_{\Omega_{n}}\nabla (
\ol\lambda_{1,n}x_2+ \ol\lambda_{2,n}\widetilde\psi_n
)\cdot\nabla(
\psi_{\Om_n,\xi_n,\zeta_n}-\ol\psi_{n})
\, dx
\\&\stackrel{\eqref{eq: clearly}}
=&\int_{\Omega_{n}}|\nabla \psi_{\Om_n,\xi_n,\zeta_n}|^2dx-
\int_{\Omega_{n}}|\nabla \ol\psi_{n} |^2dx
-2\cdot 0
\\&&-2\ol\lambda_{1,n}\{I(\Om_n,\xi_n,\zeta_n)-\mu_n\}
-2\ol\lambda_{2, n}\{C(\Om_n,\xi_n,\zeta_n)-\nu_n\}
\\&=&\int_{\Omega_{n}}|\nabla \psi_{\Om_n,\xi_n,\zeta_n}|^2dx-
\int_{\Omega_{n}}|\nabla \ol\psi_{n} |^2dx
\\&=&2\{\sL(\Om_n,\xi_n,\zeta_n)-\sL(\Om_n,\ol\xi_{n},\zeta_n)\}
\rightarrow 0
\end{eqnarray*}
because $\limsup_{n\rightarrow \infty}\sL(\Om_n,\xi_n,\zeta_n)\leq \inf_V\sL$,
$\lim_{n\rightarrow \infty}\sL(\Om_n,\ol\xi_n,\zeta_n)= \inf_V\sL$
by \eqref{eq: limsup is limit},
and $\ol\psi_n- \ol\lambda_{1,n}x_2- \ol\lambda_{2,n}\widetilde\psi_n$
has zero boundary data so it may be treated as a test function.
As a further consequence,
$\lim_{n\rightarrow \infty}\sL(\Om_n,\xi_n,\zeta_n)= \inf_V \sL$.
Hence
the distance dist$_0$ of
$\{(\Om_n,\xi_n,\zeta_n)\}$ to $D(\mu,\nu,\zeta_Q)$ tends to $0$, 
which is a contradiction.

We have assumed $\Om_n\neq \Om_Q$ for all $n\in\NN$. If
$\Om_n\neq \Om_Q$ for all $n\in\NN$ sufficiently large, 
the argument is the same. On the other hand if
$\Om_n= \Om_Q$ for infinitely many $n$, we can assume by extracting 
a subsequence that $\Om_n =\Om_Q$ for all $n\in\NN$. This case leads to a 
contradiction as follows, and therefore cannot occur.
Taking a further subsequence if needed, we can assume that
$\zeta_n\rightharpoonup \zeta$ weakly in $L^2(\Om_Q)$ and 
$\psi_{\Om_Q,\xi_n,\zeta_n}\rightarrow \psi$ in $H^1(\Om_Q)$ for some
$\zeta$ and $\psi$. We get $\zeta \in 
\overline{\sR(\Om_Q)}^w$ (see the footnote \ref{footnote}),  
$C(\Om_Q,\xi,\zeta)=\mu$
and $I(\Om_Q,\xi,\zeta)=\nu$, where $\xi$ is the trace of $\psi$.
Hence $(\Om_Q,\xi,\zeta)\in V$, which is in contradiction with (M1).
\end{proof}

We now let $t$ denote time and prove the following stability result,
after first giving a definition.

\noindent
{\bf Definition: regular flow.}
Given $\tm\in(0,\infty]$,
we call $\{\Om(t),\xi(t),\zeta(t)\}_{t \in [0,\tm)}$ a \emph{regular flow}
if, for all $t$, $\Om(t) \in \kO$,
$\xi(t) \in H^{1/2}_{per}(\sS(t))$ with
$\sS(t)=\partial \Om(t) \setminus ((0,P)\times\{0\})$,
$\zeta(t) \in L^2(\Om(t))\subset L^2((0,P)\times (0,\infty))$
  and there exists a  stream function 
$\psi \in L^\infty((0,\tm),H^2_{per}((0,P)\times 
(0,\infty)))$  
\footnote{
By definition of
this space, $\psi\in L^1_{loc}((0,\tm)\times (0,P)\times 
(0,\infty))$ and,
for almost all $t$, $\psi(t,\cdot)\in
H^2_{per}((0,P)\times (0,\infty))$. Moreover
all the derivatives up to order $2$
with respect to $x_1$ and $x_2$ are in
$L^1_{loc}((0,\tm)\times (0,P)\times 
(0,\infty))$ and the function
$t\rightarrow ||\psi(t,\cdot)||_{H^2((0,P)
\times (0,\infty))}$ is in $L^\infty$.
 }
such that
$\psi(t)=\psi(t,\cdot)|_{\Om(t)}$ is a solution to 
\eqref{fbvp}(a--d) 
for almost all $t\in[0,\tm)$. 
Let
$\psi$ give rise to the velocity field $u
=(\partial_{x_2}\psi,-\partial_{x_1}\psi)$ 
on $(0,\tm)\times (0,P)\times (0,\infty)$. 
Concerning the dependence of the
domain $\Omega(t)$ on $t$, we suppose
that $\bigcup_{t\in[0,\tm)}\Om(t)$ is bounded,
we let  $\widetilde \chi(t)$ be the characteristic function of 
$\Omega(t)$, and we assume that
the mapping $t\rightarrow \widetilde \chi 
(t) \in L^2((0,P)\times (0,\infty))$ is
continuous on $[0,\tm)$
and that  $\widetilde\chi\in L^\infty((0,\tm)\times(0,P)
\times(0,\infty))$ satisfies
the linear transport equation
$$\partial_t \widetilde\chi + \div(\widetilde\chi u) = 0
~\text{ on }~(0,\tm)\times \RR\times(0,\infty)$$
(in the sense of distributions, where $\widetilde \chi$ and $u$
are extended periodically in $x_1$).
In addition 
the mapping $t\rightarrow \zeta(t) \in L^2((0,P)\times(0,\infty))$
is supposed continuous on $[0,\tm)$
and $u$
satisfies
the time-dependent hydrodynamic problem
(Euler equation or vorticity equation), which takes the form of
convection of 
$\zeta=-\widetilde \chi \Delta \psi$ by $u$ according to
$$
\partial_t \zeta + \div(\zeta u) = 0 
$$
(in the same sense as above). 
Finally
$\mathcal{L}, I$ and $C$ are all assumed
to be conserved, that is, at all
$t\in(0,\tm)$ they have the same values as at $t=0$.

For smooth functions
these conditions are weaker 
than those of
the full evolutionary problem,
for we do not need
to be more precise in the statement of the following theorem.

Our main stability result now follows.
Whilst this is formulated in terms of $\text{dist}_0$, the subsequent
Remarks will discuss alternatives to $\text{dist}_0$ which some readers
may consider to be more natural.

\begin{theorem}
For all $\epsilon>0$, there exists  $\delta>0$ such that if
$$(\Om_0,\xi_0,\zeta_0)\in W,~~\sL(
\Om_0,\xi_0,\zeta_0)<\delta+\min_V \sL,$$
 $$\text{dist}_{L^2((0,P)\times (0,\infty))}
 \left(\zeta_0,\ol{\sR(\Omega_0)}^w
\right)<\delta ,~~
|C(\Om_0,\xi_0,\zeta_0)-\mu|<\delta,~~
|I(\Om_0,\xi_0,\zeta_0)-\nu|<\delta,$$
and if 
$$t\rightarrow (\Om(t),\xi(t), \zeta(t))\in  W $$
is a regular flow on the time interval $[0,\tm)$ such that
$(\Om(0),\xi(0), \zeta(0))=(\Om_0,\xi_0,\zeta_0)$ 
(for some $\tm\in(0,\infty]$ that is not prescribed), then
$$
\Om(t)\neq \Om_Q \hbox{ and }
\hbox{dist}_0\Big(\big(\Om(t),\xi(t),\zeta(t)\big),D(\mu,\nu,\zeta_Q)
\Big)<\epsilon$$
for all $ t\in[0,\tm)$,
\end{theorem}
\begin{proof}
If not, there exist $\epsilon>0$ and, for each $n$, a regular flow
$\{\Om_{n}(t),\xi_{n}(t),\zeta_{n}(t)\}_{t \in [0,\tm_n)}$ such that
$$\sL(\Om_{n}(0),\xi_{n}(0),\zeta_{n}(0))<\frac 1 n +\min_V\sL,
~~ \text{dist}_{L^2((0,P)\times (0,\infty))}
\left(\zeta_{n}(0),\ol{\sR(\Omega_0)}^w
\right)<\frac 1 n ,$$
$$|C(\Om_{n}(0),\xi_{n}(0),\zeta_{n}(0))-\mu|<\frac 1 n,~~
|I(\Om_{n}(0),\xi_{n}(0),\zeta_{n}(0))-\nu|<\frac 1 n$$
and $t_n \in [0,\tm_n)$ such that
either $\Om_n(t_n)=\Om_Q$ or
$$\text{dist}_0((\Om_{n}(t_n),\xi_{n}(t_n),\zeta_{n}(t_n)),D(\mu,\nu,\zeta_Q))\geq \epsilon.$$
Therefore
\begin{eqnarray*}
\sL(\Om_{n}(t_n),\xi_{n}(t_n),\zeta_{n}(t_n))
&=&\sL(\Om_{n}(0),\xi_{n}(0),\zeta_{n}(0)), \\
C(\Om_{n}(t_n),\xi_{n}(t_n),\zeta_{n}(t_n))
&=&C(\Om_{n}(0),\xi_{n}(0),\zeta_{n}(0)), \\
I(\Om_{n}(t_n),\xi_{n}(t_n),\zeta_{n}(t_n))
&=&I(\Om_{n}(0),\xi_{n}(0),\zeta_{n}(0)).
\end{eqnarray*}

We get
$$
\text{dist}_{L^2((0,P)\times (0,\infty))}
\left(\zeta_{n}(t_n),\ol{\sR}^w
\right)<\frac 1 n ;$$
to see this, we introduce as in \cite{grbarma} a ``follower''
$\chi_n(t) \in \ol{\sR}^w$ for $\zeta_n(t)$ as follows.
For each $n\in\NN$ choose 
$\chi_n(0) \in \ol{\sR(\Omega_n(0))}^w\subset\ol{\sR}^w$ with
$|| \chi_n(0)-\zeta_n(0) ||_{L^2(\Om_n(0))} < 1/n$
and let $t\rightarrow \chi_n(t)\in L^2((0,P)\times(0,\infty))$
 be the unique solution 
of the linear transport equation 
$\partial_t\chi_n+\text{div}_x(\chi_n u_n)=0$
that is continuous in $t\in[0,\tm_n)$ (with periodicity condition in $x_1$),
where
the velocity $u_n(t)$, 
as envisaged in the definition of regular flow, is assumed to lie in
$L^\infty((0,\tm_n),H^1_{per}((0,P)\times(0,\infty)))$.

The results of DiPerna and Lions \cite{DiPL} and of
Bouchut \cite{Bo} guarantee  that, for all $t\in(0,\tm_n)$,
$\chi_n(t)$ and $\zeta_n(t)$ are convected by the incompressible
flow and thus are rearrangements of 
$\chi_n(0)$ and $\zeta_n(0)$ respectively
vanishing 
outside $\Omega_n(t)$. 
See the Appendix
for a brief account of the theory in \cite{DiPL,Bo} 
that is needed on transport
equations, and in particular for 
the existence and uniqueness of $\chi_n$.

As in \cite{grbarma} we have $\chi_n(t) \in \ol{\sR}^w$ and
$\chi_n - \zeta_n$ is a solution of the transport equation, so
$$||\chi_n(t_n)-\zeta_n(t_n)||_{L^2(\Omega_n(t))}
=||\chi_n(0)-\zeta_{n}(0)||_{L^2(\Omega_n(0))}<1/n.$$
If $\Om_n(t_n)=\Om_Q$ for infinitely many $n$, we would get
a contradiction with the previous lemma.
If $\Om_n(t_n)=\Om_Q$ for finitely many $n$, 
the fact that, for large $n$, 
$(\Om_{n}(t_n),\xi_{n}(t_n),\zeta_{n}(t_n))$ 
stays away from $D(\mu,\nu,\zeta_Q)$ (with respect to dist$_0$)
would again lead to a contradiction with the previous lemma.
\end{proof}

{\bf Remarks.} 
{\bf 1.}
In the statement, the hypotheses
$$\sL(\Omega_0,\xi_0,\zeta_0) <\delta+\min_V \sL,~
|C(\Om_0,\xi_0,\zeta_0) -\mu|<\delta,~
|I(\Om_0,\xi_0,\zeta_0)-\nu|<\delta$$
can be replaced by
$$
\Om_0\neq \Om_Q~\hbox{ and }~
\text{dist}_0\Big((\Om_0,\xi_0,\zeta_0),D(\mu,\nu,\zeta_Q)\Big)
<\delta$$ 
because
$$\sL(\Omega_0,\xi_0,\zeta_0)\rightarrow \min_V \sL,~
C(\Om_0,\xi_0,\zeta_0) \rightarrow \mu,~
I(\Om_0,\xi_0,\zeta_0) \rightarrow \nu$$
as
$\text{dist}_0((\Om_0,\xi_0,\zeta_0),D(\mu,\nu,\zeta_Q))\rightarrow 0$.

{\bf 2.} Solutions to
the evolutionary problem that are considered are supposed
regular enough, but nothing is claimed about their existence.
This is why the stability result is said to be ``conditional''.
The choice of the distance in the statement is crucial for
its meaning. Conditional stability is here with respect to the
distance dist$_0$, that is, the distance dist$_0$ to the set
of minimizers is controlled for subsequent times if it is
well enough controlled initially. However nothing is said
about other distances and it could be that some other significant
distance blows up whereas dist$_0$ remains under control; as a consequence
the solution would nevertheless
cease to exist in the considered functional space. 
On the other hand, a control on
dist$_0$ could be the starting point of a  well-posedness analysis
(well-posedness of the Cauchy problem for related settings is discussed
in many papers, see e.g. \cite{CoSc}).

{\bf 3.} In the statement of the theorem, 
dist$_0$ can be replaced by the simpler distance 
\begin{eqnarray*}
&&\text{dist}_1((\Om_1,\xi_1,\zeta_1),(\Om_2,\xi_2,\zeta_2))
=\inf_{s\in[0,P]}||p_1(s+\cdot)-p_2||_{H^2_{per}}
\\&&+||\zeta_{1}-\zeta_2||_{(H^{1}((0,P)\times \RR))'} 
+||\nabla\psi_{\Om_1,\xi_1,\zeta_1}
-\nabla\psi_{\Om_2,\xi_2,\zeta_2}||_{L^2((0,P)\times (0,R))}~,
\end{eqnarray*}
where  
$\nabla\psi_{\Om_i,\xi_i,\zeta_i}$ has been trivially extended
on $((0,P)\times(0,R))\backslash \Om_i$
(thus $\mbox{dist}_1$ is defined in terms of vorticity and velocity).
Indeed, for all $\epsilon_1>0$, there exists $\epsilon_0>0$ such that,
for all $(\Om,\xi,\zeta)\in W^*$,
$$
\hbox{dist}_0\Big((\Om,\xi,\zeta)\,,\,D(\mu,\nu,\zeta_Q)\Big)<\epsilon_0
\Rightarrow
\hbox{dist}_1\Big((\Om,\xi,\zeta)\,,\,D(\mu,\nu,\zeta_Q)\Big)<\epsilon_1~.$$
Otherwise there would exist $\epsilon_1>0$ and two sequences
$\{(\Om_{1,n},\xi_{1,n},\zeta_{1,n})\}\subset W^*$ and
$\{(\Om_{2,n},\xi_{2,n},\zeta_{2,n})\}\subset D(\mu,\nu,\zeta_Q)$ 
such that
$$\lim_{n\rightarrow \infty}\hbox{dist}_0
\Big((\Om_{1,n},\xi_{1,n},\zeta_{1,n})\,,\,
(\Om_{2,n},\xi_{2,n},\zeta_{2,n})\Big)=0$$
and
\begin{multline*}
\inf_{n\in\NN}\hbox{dist}_1\Big((\Om_{1,n},\xi_{1,n},\zeta_{1,n})
\,,\,(\Om_{2,n},\xi_{2,n},\zeta_{2,n})\Big)
\\\geq \inf_{n\in\NN}\hbox{dist}_1\Big((\Om_{1,n},\xi_{1,n},\zeta_{1,n})
\,,\,D(\mu,\nu,\zeta_Q)\Big)
\geq \epsilon_1~.
\end{multline*}
Taking subsequences if necessary, we can assume that
$$\lim_{n\rightarrow \infty}\hbox{dist}_0\Big((\Om_{2,n},\xi_{2,n},\zeta_{2,n})
\,,\,(\Om,\xi,\zeta)
\Big)=0$$
for some 
$(\Om,\xi,\zeta) \in D(\mu,\nu,\zeta_Q)$  and thus
$$\lim_{n\rightarrow \infty}\hbox{dist}_0\Big((\Om_{1,n},\xi_{1,n},\zeta_{1,n})
\,,\,(\Om,\xi,\zeta)
\Big)=0.$$
If $\Omega_{1,n}=\Omega$ and 
\begin{multline*}
\nabla\overline{\psi}(\Omega_{1,n},\zeta_{1,n},C(\Omega_{1,n},\xi_{1,n},\zeta_{1,n}),I(\Omega_{1,n},\xi_{1,n},\zeta_{1,n}))
\\ = \nabla\overline{\psi}(\Omega,\zeta,C(\Omega,\xi,\zeta),I(\Omega,\xi,\zeta))
\end{multline*}
then 
\[
\hbox{dist}_1\Big((\Om_{1,n},\xi_{1,n},\zeta_{1,n})
\,,\,(\Om,\xi,\zeta)
\Big)
=
\hbox{dist}_0\Big((\Om_{1,n},\xi_{1,n},\zeta_{1,n})
\,,\,(\Om,\xi,\zeta)
\Big),
\]
whereas otherwise,
\begin{multline*}
\| \nabla\psi_{\Omega_{1,n},\xi_{1,n},\zeta_{1,n}}
- \nabla\psi_{\Omega,\xi,\zeta}\|_{L^2((0,P)\times(0,R))}
\\ \leq \| \nabla\psi_{\Omega_{1,n},\xi_{1,n},\zeta_{1,n}}
- \nabla\psi_{\Omega,\xi,\zeta}\|_{L^2(\Omega_{1,n})}
+ \|\nabla\psi_{\Omega,\xi,\zeta}\|_{L^2(\Omega\setminus\Omega_{1,n})}
\\ \leq \| \nabla\psi_{\Omega_{1,n},\xi_{1,n},\zeta_{1,n}} 
 - \nabla\overline{\psi}(\Omega_{1,n},\zeta_{1,n},C(\Omega_{1,n},\xi_{1,n},\zeta_{1,n}),I(\Omega_{1,n},\xi_{1,n},\zeta_{1,n}))\|_{L^2(\Omega_{1,n})}
\\ + \| \nabla\overline{\psi}(\Omega_{1,n},\zeta_{1,n},C(\Omega_{1,n},\xi_{1,n},\zeta_{1,n}),I(\Omega_{1,n},\xi_{1,n},\zeta_{1,n}))
\\ - \nabla\overline{\psi}(\Omega,\zeta,C(\Omega,\xi,\zeta),I(\Omega,\xi,\zeta))\|_{L^2(\Omega_{1,n})}
\\ + \| \nabla\overline{\psi}(\Omega,\zeta,C(\Omega,\xi,\zeta),I(\Omega,\xi,\zeta)) - \nabla\psi_{\Omega,\xi,\zeta}\|_{L^2(\Omega_{1,n})}
+ \| \nabla\psi_{\Omega,\xi,\zeta}\|_{L^2(\Omega\setminus\Omega_{1,n})}
\\ \leq \| \nabla\psi_{\Omega_{1,n},\xi_{1,n},\zeta_{1,n}} 
- \nabla\overline{\psi}(\Omega_{1,n},\zeta_{1,n},C(\Omega_{1,n},\xi_{1,n},\zeta_{1,n}),I(\Omega_{1,n},\xi_{1,n},\zeta_{1,n}))\|_{L^2(\Omega_{1,n})}
\\ + \| \nabla\overline{\psi}(\Omega_{1,n},\zeta_{1,n},C(\Omega_{1,n},\xi_{1,n},\zeta_{1,n}),I(\Omega_{1,n},\xi_{1,n},\zeta_{1,n}))
\\ - \nabla\overline{\psi}(\Omega,\zeta,C(\Omega,\xi,\zeta),I(\Omega,\xi,\zeta))\|_{L^2((0,P)\times(0,R))}
\\ + \| \nabla\overline{\psi}(\Omega,\zeta,C(\Omega,\xi,\zeta),I(\Omega,\xi,\zeta)) - \nabla\psi_{\Omega,\xi,\zeta}\|_{L^2(\Omega)}
\\ +   \| \nabla\overline{\psi}(\Omega,\zeta,C(\Omega,\xi,\zeta),I(\Omega,\xi,\zeta)) \|_{L^2(\Omega_{1,n}\setminus\Omega)}
+ \| \nabla\psi_{\Omega,\xi,\zeta}\|_{L^2(\Omega\setminus\Omega_{1,n})} ,
\end{multline*}
and hence
\begin{multline*}
\hbox{dist}_1\Big((\Om_{1,n},\xi_{1,n},\zeta_{1,n})
\,,\,(\Om,\xi,\zeta)
\Big)
\leq
\hbox{dist}_0\Big((\Om_{1,n},\xi_{1,n},\zeta_{1,n})
\,,\,(\Om,\xi,\zeta)
\Big)
\\ + \| \nabla\overline{\psi}(\Omega,\zeta,C(\Omega,\xi,\zeta),I(\Omega,\xi,\zeta)) \|_{L^2(\Omega_{1,n}\setminus\Omega)}
+ \| \nabla\psi_{\Omega,\xi,\zeta}\|_{L^2(\Omega\setminus\Omega_{1,n})} .
\end{multline*}
Thus in either case we get the contradiction
$$0=
\lim_{n\rightarrow \infty}\hbox{dist}_1\Big((\Om_{1,n},\xi_{1,n},\zeta_{1,n})
\,,\,(\Om,\xi,\zeta)
\Big)\geq \epsilon_1.$$

{\bf 4.} In the statement of the theorem, 
the hypotheses
$$\sL(\Omega_0,\xi_0,\zeta_0) <\delta+\min_V \sL,~
|C(\Om_0,\xi_0,\zeta_0) -\mu|<\delta,~
|I(\Om_0,\xi_0,\zeta_0)-\nu|<\delta$$
can also be replaced by
$$
\text{dist}_1\Big((\Om_0,\xi_0,\zeta_0),D(\mu,\nu,\zeta_Q)\Big)
<\delta$$ 
(compare with the first remark).
Indeed, for all $\delta>0$, there exists $\delta_1>0$ such that
all $(\Om,\xi,\zeta)\in W$ satisfying
$
\hbox{dist}_1\Big((\Om,\xi,\zeta)\,,\,D(\mu,\nu,\zeta_Q)\Big)<\delta_1
$
also satisfy
$$\sL(\Omega,\xi,\zeta)<\delta+ \min_V \sL,~
|C(\Om,\xi,\zeta) -\mu|<\delta,~
|I(\Om,\xi,\zeta)- \nu|<\delta.$$
Otherwise there would exist $\delta>0$ and two sequences
$\{(\Om_{1,n},\xi_{1,n},\zeta_{1,n})\}\subset W$ and
$\{(\Om_{2,n},\xi_{2,n},\zeta_{2,n})\}\subset D(\mu,\nu,\zeta_Q)$ 
such that
$$\lim_{n\rightarrow \infty}\hbox{dist}_1
\Big((\Om_{1,n},\xi_{1,n},\zeta_{1,n})\,,\,
(\Om_{2,n},\xi_{2,n},\zeta_{2,n})\Big)=0$$
and such that one of the following inequalities holds:
\begin{multline*}
\inf_{n\in\NN}
\sL(\Omega_{1,n},\xi_{1,n},\zeta_{1,n})\geq \delta+ \min_V \sL,~
\\ \inf_{n}
|C(\Om_{1,n},\xi_{1,n},\zeta_{1,n}) -\mu|\geq \delta,~
\inf_{n}
|I(\Om_{1,n},\xi_{1,n},\zeta_{1,n})- \nu|\geq \delta.
\end{multline*}
Taking subsequences if necessary, we can assume that
$$\lim_{n\rightarrow \infty}\hbox{dist}_0\Big((\Om_{2,n},\xi_{2,n},\zeta_{2,n})
\,,\,(\Om,\xi,\zeta)
\Big)=0$$
for some 
$(\Om,\xi,\zeta) \in D(\mu,\nu,\zeta_Q)$. Arguing as above, we get
$$\lim_{n\rightarrow \infty}\hbox{dist}_1\Big((\Om_{2,n},\xi_{2,n},\zeta_{2,n})
\,,\,(\Om,\xi,\zeta)
\Big)=0$$
  and thus
$$\lim_{n\rightarrow \infty}\hbox{dist}_1\Big((\Om_{1,n},\xi_{1,n},\zeta_{1,n})
\,,\,(\Om,\xi,\zeta)
\Big)=0.$$
We then get the contradiction
$$
\sL(\Omega_{1,n},\xi_{1,n},\zeta_{1,n})\rightarrow  \min_V \sL,~
C(\Om_{1,n},\xi_{1,n},\zeta_{1,n}) \rightarrow \mu,~
I(\Om_{1,n},\xi_{1,n},\zeta_{1,n})\rightarrow  \nu
$$
(see \eqref{eq: computation I} and
 \eqref{eq: computation II} for similar computations).

\section{Appendix: transport equation theory needed to construct the follower}

During the proof of Theorem 5.2 we introduce a ``follower'', in
$\overline{\sR}^w$, of a regular flow, by convecting a suitable element
of $\overline{\sR}^w$ using the velocity field of the flow. Here we present
the theory of transport equations needed to justify this construction.

Let us consider a regular flow (see the above definition).
As $\bigcup_{t\in [0,\tm)}\Omega(t)$ is bounded, we can suppose that,
for some $R>0$, 
$\bigcup_{t\in [0,\tm)}\Omega(t)\subset (0,P)\times (0,R)$ 
and the divergence-free velocity 
$u\in L^\infty((0,\tm),H^1_{per}((0,P)\times(0,\infty)))$
vanishes for $x_2>R$. 
We extend $u$ 
to all of $\RR \times \RR^2$ by setting $u(t,x_1,x_2)=0$ for $t\not\in[0,\tm)$,
$u(t,x_1,x_2)=(u_1(t,x_1,-x_2),-u_2(t,x_1,-x_2))$ for $x_2<0$
and by $P$-periodicity in $x_1$.
We use the notation $u=(u_1,u_2)$ and $u(t)=u(t,\cdot)$.
As, for almost all $t$, the trace of $u_2(t)$ 
on the set $x_2=0$ is trivial (see \eqref{bottom bc}), 
$u$ is now well defined in
$L^\infty(\RR,H^1_{per}(\RR^2))$  and
still divergence free.

\subsubsection*{Existence}

Consider initial data $\chi(0) \in L^2 (\Omega(0))\subset
L^2((0,P)\times (0,\infty))$ and
extend it periodically in $x_1$ so that we can see it in
$L^2_{per}(\RR\times (0,\infty))\subset L^2_{per}(\RR^2)$
(and $\chi(0)$ vanishes when $x_2<0$).
Mollify $\chi(0)$ in $x$ to get $\chi_\varepsilon(0)$ and
mollify $u$ in $x$ and $t$ to get $u_{\varepsilon,\tau}(t)$ bounded
in $H^1_{per}(\RR^2)$. This can be done in such a way
that the second component of
$u_{\varepsilon,\tau}(t)$ vanishes on $x_2=0$.
Since, for fixed $\epsilon$ and $\tau$,
$u_{\varepsilon,\tau}\in L^\infty(\RR \times \RR^2)$,
the solution of
$$\partial_t \chi + \div(\chi u_{\varepsilon,\tau})=0
~ \text{ in }~ [0,\infty)\times \RR^2$$
with initial data $\chi_\varepsilon(0)$ exists for all positive
time by using the flow
of $u_{\varepsilon,\tau}$; denote it $\chi_{\varepsilon,\tau}(t)
\in L^2_{per}(\RR^2)$.
Notice that $u_{\varepsilon,\tau}(t)$ is still divergence-free
and
therefore the flow is rearrangement-preserving, hence
$$\| \chi_{\varepsilon,\tau}(t) \|_2=
\| \chi_{\varepsilon,\tau}(t) \|
_{L^2((0,P)\times\RR)} =
\| \chi_{\varepsilon}(0) \|_2
\leq \| \chi(0)\|_2 .
$$

Then, for any $1<s<2$, we have
\begin{multline*}
\|\chi_{\varepsilon,\tau}(t) u_{\varepsilon,\tau}(t)\|_s
=\|\chi_{\varepsilon,\tau}(t) 
u_{\epsilon,\tau}(t)\|_{L^s((0,P)\times\RR)}
\leq\|\chi_{\varepsilon,\tau}(t)\|_2\|u_{\varepsilon,\tau}(t)\|_{2s/(2-s)}
\\
\leq \|\chi(0)\|_2\|u(t)\|_{H^1},
\end{multline*}
so we have $\chi_{\varepsilon,\tau}(t) u_{\varepsilon,\tau}(t)$ bounded in
$L^s$ and thus
$\div(\chi_{\varepsilon,\tau}(t) u_{\varepsilon,\tau})$ bounded in 
$W^{-1,s}$. Hence, as in Lemma 10 in \cite{grbarma}, for
$0\leq t_1<t_2$,
\[
\| \chi_{\varepsilon,\tau}(t_2) - \chi_{\varepsilon,\tau}(t_1) \|_{-1,s}
\leq M \| \chi(0) \|_2 |t_2-t_1|
\]
where $M$ is a bound on $\| u(t) \|_{H^1}$ for almost all
 $t\in[t_1,t_2]$.

Let $1/r+1/s=1$ (so $2<r<\infty$).
Then 
$W^{1,r}((0,P)\times(-2R,2R)) \hookrightarrow L^2((0,P)\times (-2R,2R))$ 
compactly and, taking the adjoints,
$L^2 \hookrightarrow W^{-1,s}$ compactly.
Since the $\chi_{\varepsilon,\tau}(t)$ all lie in a ball in 
$L^2((0,P)\times(-2R,2R))$ (for $\epsilon,\tau$ small enough)
and hence
lie in a strongly compact set in $W^{-1,s}$, we can apply 
the Arzel\`{a}-Ascoli theorem to
let $\varepsilon,\tau \to 0$ (along any particular sequences) 
and obtain a sequence converging in
$L^\infty((0,\tm),W^{-1,s}_{per}(\RR\times(-2R,2R)))$ 
and weakly in $L^2$ on any bounded open subset of
$(0,\tm)\times \RR^2$ 
to a limit 
$$\chi\in C([0,\tm),W^{-1,s}_{per}(\RR^2))
\cap L^2_{loc}((0,\tm)\times \RR^2)
\cap L^\infty((0,\tm),L^2_{per}(\RR^2)),$$
where $L^2_{per}(\RR^2)$ is endowed with the norm
of $L^2((0,P)\times \RR)$. Moreover $\chi$
solves the linear transport equation on $(0,\tm)\times \RR^2$
with initial condition $\chi(0)$,
$\chi(t)$ is also weakly
continuous in $L^2$ with respect to $t\in[0,\tm)$,
$\chi(t)$ vanishes for $x_2<0$ and for $x_2>R$
and $\chi(t)\geq 0$ if $\chi(0)\geq 0$
(because of the way $\chi(\cdot)$ has been obtained as a limit;
remember that $\chi(0)$ vanishes for $x_2\not\in[0,R]$,
and since $u(t)$ vanishes for $x_2>R$ and the second component
of $u(t)$ is odd in $x_2$ it follows that the trajectories of the
approximating flows do not cross the lines $x_2=0$ and $x_2=R+\varepsilon$). 

\subsubsection*{Rearrangement and uniqueness}
Let $t\rightarrow \chi(t) \in L^2_{per}(\RR\times (0,\infty))$
be such that 
\begin{multline*}
\chi\in C([0,\tm),W^{-1,s}_{per}(\RR\times(0,\infty)))
\cap L^2_{loc}((0,\tm)\times \RR\times (0,\infty))
\\ \cap L^\infty_{loc}((0,\tm),L^2_{per}(\RR\times(0,\infty))),
\end{multline*}
the support of $\chi$ is uniformly bounded in the $x_2$ direction
and $\chi$ satisfies the linear transport
equation  on $(0,\tm)\times \RR\times (0,\infty)$, that is,
\begin{equation}\label{eq: lte}
\int_{(0,\tm)\times\RR\times(0,\infty)}
(\partial_t \varphi+\nabla \varphi\cdot u)\chi\, dtdx=0
\end{equation}
for all $\varphi\in \dsr((0,\tm)\times \RR\times (0,\infty))$.
Here  $\chi$ is not necessarily restricted
to be the solution obtained just above and,
provided that
$\chi \in L^2_{loc}((0,\tm)\times \RR\times (0,\infty))
\cap L^\infty_{loc}((0,\tm),L^2_{per}(\RR\times(0,\infty)))$,
the hypothesis
$\chi\in C([0,\tm),W^{-1,s}_{per}(\RR\times(0,\infty)))$ above
is equivalent in this context to the requirement that
$t\rightarrow \chi(t)\in L^2_{per}(\RR\times (0,\infty))$ 
is continuous in $t\geq 0$ with respect to the weak topology on
$L^2_{per}(\RR\times (0,\infty))$.

Let us check that \eqref{eq: lte} still holds for all
$\varphi\in \dsr((0,\tm)\times \RR^2)$,  so that
$\chi$ is also a solution to 
the linear transport equation  on 
$(0,\tm)\times \RR^2$ (where $\chi$ vanishes
if $x_2<0$). Given such a $\varphi$, 
we introduce $f\in C^\infty(\RR)$ such that $f(x_2)=0$
for $x_2\leq 0$, $f(x_2)=1$ for $x_2\geq 1$ and $f$
is increasing. We set $f_\delta(x_2)=f(x_2/\delta)$ and
observe that
$$\int_{(0,\tm)\times\RR\times(0,\infty)}
(f_\delta(x_2)\partial_t \varphi +f_\delta(x_2)\nabla \varphi\cdot u
+f_\delta'(x_2)\varphi u_2)\chi\, dtdx=0,$$
where $u=(u_1,u_2)$. As
$\chi\in L^\infty((0,\tm),L^2_{per}(\RR\times(0,\infty)))$, we get
$$\int_{(0,\tm)\times\RR\times(0,\infty)}
(f_\delta(x_2)\partial_t \varphi
+f_\delta(x_2)\nabla \varphi\cdot u) \chi\, dt dx
\to \int_{(0,\tm)\times \RR^2}(\partial_t \varphi 
+\nabla \varphi\cdot u) \chi\, dt dx
$$
as $\delta\to 0$, by Lebesgue's theorem.
Moreover, if $\varphi$ is supported in $(0,T)\times (-A,A)^2$ with
$0<T<\tm$, then
\begin{eqnarray*}
&&\left| 
\int_{(0,\tm)\times\RR\times(0,\infty)}
f_\delta'(x_2)\varphi u_2 \chi\, dtdx\right| 
\\&\leq &\delta^{-1}\text{const}\,
\|\varphi \chi \|_{L^2((0,T)\times (-A,A)\times (0,\delta))}
\| u_2 \|_{L^\infty((0,T),L^2((-A,A)\times (0,\delta)))} \\
&\stackrel{\text{Poincar\'e}}{\leq} &
\text{const}\,
\|\varphi \chi \|_{L^2((0,T)\times (-A,A)\times (0,\delta))}
\|\nabla u_2 \|_{L^\infty((0,T),L^2((-A,A)\times (0,\delta)))} \\
& \to& 0
\end{eqnarray*}
as $\delta \to 0$, because
$\chi\in  L^2_{loc}((0,\tm)\times \RR\times (0,\infty))$ 
 (Poincar\'e's inequality
is available thanks to the fact that the trace of $u_2(t)$
on $x_2=0$ vanishes for almost all $t$; see e.g. \cite{AdFo}, 
sect. 6.26 in the 1st edition or 6.30 in 
the 2nd).
Thus \eqref{eq: lte} holds for the more general $\varphi$ as desired.

Now that we know that
$$
\chi\in C([0,\tm),W^{-1,s}_{per}(\RR^2))
\cap L^2_{loc}((0,\tm)\times \RR^2)
\cap L^\infty_{loc}((0,\tm),L^2_{per}(\RR^2))
$$
is a solution to the linear transport
equation on $(0,\tm)\times \RR^2$, 
we mollify in $x$ to get 
$\chi_\varepsilon\in C([0,\tm),L^\infty_{per}(\RR^2))$. 
We also assume that $\chi$ vanishes if $x_2\not \in [0,R]$.

Choose any $T\in(0,\tm)$.
Then, for bounded $g \in C^1(\RR)$, by Bouchut \cite{Bo},
proof of Thm 3.2(ii) (especially Lemma 3.1(ii) applied
to eq. (3.23)), we have
$$\partial_t g(\chi_\varepsilon) + \div(g(\chi_\varepsilon) u)
= r_\varepsilon \to 0 \mbox{ in }
L^1((0,T),L^1_{loc}(\RR^2)) \mbox{ as } \varepsilon \to 0.$$
Integrating against a smooth test function of 
the form $h(t)f(x)$ we have
\begin{multline}
\left|\int_{\RR^3} h^\prime f g(\chi_\varepsilon)dtdx +
\int_{\RR^3} h \nabla f \cdot u g(\chi_\varepsilon)dtdx\right| =
\left|\int_{\RR^3} h f r_\varepsilon dt dx\right|
\\
\leq
\| r_\varepsilon \|_{L^1((0,T)\times(-P,2P)\times (-2R,2R) )}
\label{approxeq}
\end{multline}

provided 
$\sup_{t\in\RR}|h(t)|\leq 1$,  $\sup_{x\in\RR^2}| f(x)|\leq 1$,
$h$ is compactly supported in $(0,T)$ and $f$ is
compactly supported in $(-P,2P)\times (-2R,2R)$.

Choose $f=f_\delta\in \dsr(\RR^2)$
of the form $f_\delta(x_1,x_2)=f_1(x_1)f_2(x_2)$ where $f_1$
vanishes outside $[0,P+\delta]$
and is identically equal to $1$ on $[\delta,P]$,
while $f_2$ is compactly supported in 
$(-R-\delta,R+\delta)$
and is identically equal
to $1$ on $[-R, R]$.
We assume
$0<\delta<\min\{P/2,R\}$. 
By approximations, the class of allowed
$f_1$ can be enlarged to continuous functions that are piecewise
$C^1$, and therefore we can choose $f_1$ such that
$f_1(x_1)=x_1/\delta$ on $[0,\delta]$ and 
$f_1(x_1)=1-(x_1-P)/\delta$ on $[P,P+\delta]$. 
Then
\[
\int_{\RR^2} g(\chi_\varepsilon) \nabla f_\delta \cdot u \, dx
= \int_{\RR^2} g(\chi_\varepsilon) f_1^\prime(x_1) f_2(x_2) u_1
\, dx ,\]
because $u=(u_1,u_2)$ vanishes if $x_2\not \in [-R,R]$
and thus $f_2'(x_2)u_2$ vanishes almost everywhere on $\RR^2$,
where  $t$ is fixed in a set of full measure
in $(0,T)$.
The contributions to the integral of the regions
$[0,\delta]\times [-2R,2R]$ and
$[P,P+\delta] \times [-2R,2R]$ 
are equal and opposite (because $\chi_\varepsilon$ and
$u$ are  $P$-periodic
in $x_1$, and $f'(x_1)=\pm 1/\delta$ there), while
$g(\chi_\varepsilon) f_1^\prime(x_1) f_2(x_2) u_1$
vanishes everywhere else.
 Hence
\begin{equation}
\int_{\RR^2} g(\chi_\varepsilon(t,x))
\nabla f_\delta(x) \cdot u(t,x) dx
= 0.
\label{limitzero}
\end{equation}

For $0< t_1<t_2< T$,
now take $h=h_\delta$ in \eqref{approxeq} to be any test function on
$(0,T)$ with
$0 \leq h_\delta \leq 1$, vanishing outside
$(t_1,t_2)$, equal to $1$ on $[t_1+\delta,t_2-\delta]$, with
$0\leq h_\delta^\prime  \leq 2/\delta$ on $(t_1,t_1+\delta)$
and $0\leq - h_\delta^\prime  \leq 2/\delta$ on $(t_2-\delta,t_2)$
($0<\delta<(t_2-t_1)/2$).
Applying \eqref{limitzero} and letting $\delta \to 0$, we obtain
\begin{eqnarray*}
&& \left| \int_{(0,P)\times (-R,R)} g(\chi_\varepsilon(t_2))dx  -  
\int_{(0,P)\times (-R,R)} g(\chi_\varepsilon(t_1))dx\right| \\
&&\leq \|  r_\varepsilon \|_{L^1((0,T) \times (-P,2P)\times (-2R,2R) )} 
\end{eqnarray*}
because
$g(\chi_\varepsilon)\in C([0,\tm),L^\infty_{per}(\RR^2))$. 
Letting $\varepsilon \to 0$ yields
\[
\int_{(0,P)\times (-R,R)}
 g(\chi(t_2))dx = \int_{(0,P)\times(-R,R)} g(\chi(t_1))dx
\]
and we deduce that $\chi(t_2)$ is a rearrangement of $\chi(t_1)$ 
in $L^2((0,P)\times (-R,R))$. As a consequence
$\chi(t_2)$ is a rearrangement of $\chi(t_1)$ 
in $L^2((0,P)\times (0,R))$
and hence
$\| \chi(t,\cdot) \|_{L^2((0,P)\times (0,\infty))}$ 
is constant in time.
As $T\in(0,\tm)$ is arbitrary,
this proves any solution 
\begin{multline*}
\chi\in C([0,\tm),W^{-1,s}_{per}(\RR\times(0,\infty)))
\cap L^2_{loc}((0,\tm)\times \RR\times (0,\infty))
\\ \cap L^\infty_{loc}((0,\tm),L^2_{per}(\RR\times(0,\infty)))
\end{multline*}
of the linear transport equation  on $(0,\tm)\times \RR
\times (0,\infty)$
such that
$\chi$ vanishes for all $x_2\not\in (0,R)$ 
is strongly continuous with respect to
$L^2_{per}(\RR\times(0,\infty))$
(because it is weakly continuous and the $L^2$-norm is preserved).
In addition $\chi(t)$ is a rearrangement of $\chi(0)$
for all $t\in(0,\tm)$ and therefore  if $\chi(0)=0$
then $\chi(t)=0$ for all $t\in(0,\tm)$.
If $\chi(0)$
is not necessarily trivial, this implies by linearity that
$t\rightarrow \chi(t)$ is unique given $\chi(0)$
(more precisely, unique in this class). 

Let $\Omega(t)$ for $t\in(0,\tm)$ and $\widetilde \chi$ be as
in the definition of a regular flow in the previous section, and
assume moreover
that $\chi(0)$ vanishes outside
$\Omega(0)$. Then $\chi^2/(1+\chi^2)\in[0,1)$ is a solution to
the linear transport equation
on $(0,\tm)\times \RR^2$ (see Thm 3.2(ii) in \cite{Bo})
and so is $\widetilde \chi-\chi^2/(1+\chi^2)$ 
(by linearity).
As
$\widetilde \chi(0)-\chi(0)^2/(1+\chi(0)^2)\geq 0$ almost
everywhere, we get 
$\widetilde \chi(t)-\chi(t)^2/(1+\chi(t)^2)\geq 0$ for all $t\in[0,\tm)$
and thus $\chi(t)$ is supported by $\Omega(t)$ for all $t\in[0,\tm)$.

\end{document}